\documentclass[11pt]{amsart}
\usepackage[colorlinks=true, pdfstartview=FitV, linkcolor=blue, 
citecolor=blue]{hyperref}

\usepackage{amsmath,amssymb,bbm,amscd}
\usepackage{graphicx}
\usepackage{a4wide}

\theoremstyle{plain}
\newtheorem{theorem}{Theorem}[section]
\newtheorem{prop}[theorem]{Proposition}
\newtheorem{lemma}[theorem]{Lemma}
\newtheorem{coro}[theorem]{Corollary}
\newtheorem{fact}[theorem]{Fact}

\theoremstyle{definition}

\newtheorem{remark}[theorem]{Remark}

\numberwithin{equation}{section}

\newcommand{\ii}{\ts\mathrm{i}}

\newcommand{\ts}{\hspace{0.5pt}}
\newcommand{\nts}{\hspace{-0.5pt}}

\DeclareMathOperator{\dens}{\mathrm{dens}}
\DeclareMathOperator{\card}{\mathrm{card}}
\DeclareMathOperator{\cent}{\mathrm{cent}}
\DeclareMathOperator{\Gal}{\mathrm{Gal}}
\DeclareMathOperator{\stab}{\mathrm{stab}}
\DeclareMathOperator{\norm}{\mathrm{norm}}
\DeclareMathOperator{\ord}{\mathrm{ord}}
\DeclareMathOperator{\vol}{\mathrm{vol}}
\DeclareMathOperator{\GL}{\mathrm{GL}}
\DeclareMathOperator{\No}{\mathrm{N}}

\DeclareMathOperator{\lcm}{\mathrm{lcm}}

\newcommand{\vG}{\varGamma}
\newcommand{\vL}{\varLambda}
\newcommand{\vO}{\varOmega}

\newcommand{\fa}{\mathfrak{a}}

\newcommand{\fp}{\mathfrak{p}}

\newcommand{\fz}{\mathfrak{z}}

\newcommand{\cB}{\mathcal{B}}

\newcommand{\cE}{\mathcal{E}}
\newcommand{\cG}{\mathcal{G}}
\newcommand{\cH}{\mathcal{H}}
\newcommand{\cL}{\mathcal{L}}
\newcommand{\cP}{\mathcal{P}}
\newcommand{\cR}{\mathcal{R}}
\newcommand{\cS}{\mathcal{S}}

\newcommand{\cO}{\mathcal{O}}

\newcommand{\AAA}{\mathbb{A}}

\newcommand{\FF}{\mathbb{F}}
\newcommand{\ZZ}{\mathbb{Z}\ts}
\newcommand{\RR}{\mathbb{R}\ts}
\newcommand{\CC}{\mathbb{C}\ts}
\newcommand{\NN}{\mathbb{N}}
\newcommand{\QQ}{\mathbb{Q}}

\newcommand{\XX}{\mathbb{X}}
\newcommand{\YY}{\mathbb{Y}}

\newcommand{\Aut}{\mathrm{Aut}}
\newcommand{\pram}{\cP^{}_{\nts\mathrm{ram}}}

\newcommand{\defeq}{\mathrel{\mathop:}=}

\newcommand{\exend}{\hfill $\Diamond$}

\newcommand{\myfrac}[2]{\frac{\raisebox{-2pt}{$#1$}}
      {\raisebox{0.5pt}{$#2$}}}
\newcommand{\bs}[1]{\boldsymbol{#1}}

\newcommand\smallO{
	\mathchoice
	{{\scriptstyle\mathcal{O}}}% \displaystyle
	{{\scriptstyle\mathcal{O}}}% \textstyle
	{{\scriptscriptstyle\mathcal{O}}}% \scriptstyle
	{\scalebox{.7}{$\scriptscriptstyle\mathcal{O}$}}%\scriptscriptstyle
}

\begin{document}

\title[On $k$-free numbers in cyclotomic fields]{On
  $k$-free numbers in cyclotomic fields:\\[2mm]
  entropy, symmetries and topological invariants}

\author{Michael Baake}

\address{Fakult\"{a}t f\"{u}r Mathematik,
  Universit\"{a}t Bielefeld,\newline \indent Postfach
  100131, 33501 Bielefeld, Germany}
\email{mbaake@math.uni-bielefeld.de}

\author{\'{A}lvaro Bustos}

\address{Facultad de Matem\'{a}ticas,
  Pontificia Universidad Cat\'{o}lica de Chile, \newline
  \indent Vicu\~{n}a Mackenna 4860, 7820436 Macul, Santiago, Chile}
\email{abustog@mat.uc.cl}

\author{Andreas Nickel}

\address{Institut f\"ur Theoretische Informatik, Mathematik und
  Operations Research, \newline \indent Universit\"at der Bundeswehr
  M\"unchen, \newline \indent Werner-Heisenberg-Weg~39, 85579
  Neubiberg, Germany} \email{andreas.nickel@unibw.de}

\begin{abstract}
  Point sets of number-theoretic origin, such as the visible lattice
  points or the $k$-th power free integers, have interesting geometric
  and spectral properties and give rise to topological dynamical
  systems that belong to a large class of subshifts with positive
  topological entropy. Among them are $\cB$-free systems in one
  dimension and their higher-dimensional generalisations, most
  prominently the $k$-free integers in algebraic number fields. Here,
  we extend previous work on quadratic fields to the class of
  cyclotomic fields. In particular, we discuss their entropy and
  extended symmetries, with special focus on the interplay between
  dynamical and number-theoretic notions.
\end{abstract}


\maketitle
\thispagestyle{empty}

\section{Introduction}\label{sec:intro}

The set $V^{}_2$ of square-free integers consists of all numbers in
$\ZZ$ that are not divisible by the square of a rational prime.  As
such, despite having gaps of arbitrary length, it has density
$6/\pi^2$ when determined with respect to centred intervals
$I_n \defeq [-n,n]$ as $n\to\infty$, which is called \emph{tied
  density} in \cite{BMP}. It is also the upper density of $V^{}_2$ as
we shall detail below. Moreover, $V^{}_{2}$ has pure-point diffraction
in the sense of measure theory as follows. The Dirac comb
\[
     \delta^{}_{V_2} \, \defeq \sum_{x\in V_2} \delta_x \ts ,
\]
where $\delta_x$ denotes the normalised Dirac measure at $x$, has a
unique autocorrelation with respect to the averaging sequence
$(I_n)^{}_{n\in\NN}$, called $\gamma^{}_{2}$. The latter is positive
definite and thus Fourier transformable, and the transform
$\widehat{\ts\gamma^{}_{2}\ts}$ is a positive pure-point measure that
is supported on the elements of $\QQ$ with cube-free denominator; see
\cite{BMP} for a proof as well as \cite{TAO,KRS} for further
background.  Analogous statements are true for the set $V^{}_{k}$ of
$k$-th power free integers, for $k\geqslant 2$, which we simply call
$k$-free integers from now on.

Likewise, the set $V\subset \ZZ^2$ of visible points consists of all
elements of $\ZZ^2$ with coprime coordinates; see the front cover of
\cite{Apo} for an illustration. Its (tied) density is once again
$6/\pi^2$, this time calculated for centred squares of growing area,
such as $[-n,n]^2$ with $n\in\NN$. Once again, it is also the maximal
upper density of $V$ with respect to any averaging sequence.  The set
$V$ has holes of arbitrary size that even repeat lattice-periodically,
but once again pure-point diffraction relative to the above averaging
sequence. Here, the diffraction measure is supported on the points of
$\QQ^2$ with square-free denominator \cite{BMP,TAO}. Clearly, one can
define the analogous object for $\ZZ^d$ and any $d\geqslant 2$, which
gives a point set of density $1/\zeta(d)$, where $\zeta$ denotes
Riemann's zeta function.

Both cases provide uniformly discrete point sets of number-theoretic
origin in Euclidean space. They give rise to interesting topological
dynamical systems that can support many invariant measures. In the
case of $V^{}_{2}$, one defines the space, or (discrete) \emph{hull},
\[
    \XX^{}_{2} \, \defeq \, \overline{ \{ t + V^{}_{2} : t \in \ZZ \} }
\]
where $t + V^{}_{2}$ is the translation of $V^{}_{2}$ by $t$ and the
bar denotes the orbit closure in the local topology. Here, two subsets
of $\ZZ$ are $(1/n)$-close if they agree on the interval $I_n$. This
topology agrees with the product topology of $\{0,1\}^{\ZZ}$ if one
identifies a subset of $\ZZ$ with an element of this product space in
the canonical way, that is, via its characteristic function (which
is the square of the M\"{o}bius function in this example). It is
then clear that $\XX^{}_{2}$ is compact, though it has a complicated
structure. In particular, it contains the empty set (or the all-$0$
configuration), which is the only minimal subset of $\XX_2$.

Since $\ZZ$ acts continuously on $\XX_{2}$ via translation, the pair
$(\XX_{2}, \ZZ)$ constitutes a \emph{topological dynamical system}, or
TDS for short.  It is clear that it can support many different
invariant probability measures.  Among them are the patch frequency
measure $\mu_{_\text{M}}$ with respect to the averaging sequence
$(I_n)^{}_{n\in\NN}$, which is also known as the Mirsky measure
\cite{BH,BBHLN,KRS}, and the Dirac measure that is concentrated on the
empty set, denoted by $\delta^{}_{\varnothing}$. The latter is a patch
frequency measure, too, this time with respect to an averaging
sequence that `chases' holes of growing size.  Both measures have a
natural interpretation in terms of weak model sets (in the sense of
Meyer \cite{Meyer,Schreiber}) of extremal density \cite{TAO,BHS,RS}, but
many other measures exist. As they do not have an interpretation
as a patch frequency measure for $V_2$ with respect to a nested
averaging sequence around the origin, they are less natural from the
viewpoint of diffraction. There are of course other measures for
which we would get pure-point spectrum (by selecting subsets that
are (mean) Besicovitch almost periodic in the sense of \cite{LSS},
such as the $k$-free integers with $k>2$), but we prefer to consider
them separately. The situation for the visible lattice points is
completely analogous \cite{BH,BBHLN}.

Once an invariant measure is specified, say $\mu$, we get a
measure-theoretic dynamical system, $(\XX_2, \ZZ, \mu)$, and may then
use the notion of the \emph{dynamical spectrum} via the spectrum of
the Koopman operator induced on the Hilbert space
$H=L^{2} (\XX_2, \mu)$.  The spectrum is pure-point when the
eigenfunctions of the Koopman operator form a basis of $H$.  As
follows from \cite{BL}, the two notions of pure pointedness are
equivalent, so the set $V^{}_{2}$ has pure-point diffraction if and
only if $(\XX_2, \ZZ, \mu)$ has pure-point dynamical spectrum, where
both statements refer to the same averaging sequence. This point of
view relates the harmonic analysis of the point set $V^{}_{2}$ with
the dynamical systems approach to it; see also \cite{KRS}.
Interestingly, the set $V^{}_{2}$ has positive patch counting entropy,
namely $6 \ts \log (2)/\pi^2$, while the Kolmogorov--Sinai entropy of
the Mirsky measure $\mu_{_\text{M}}$ vanishes, as it does for the
visible lattice points \cite{PH}.

This raises an obvious question about the nature of the measure of
maximal entropy, as studied in \cite{Peck}; see also \cite{KLW}.  It
is known that all ergodic measures are joinings of the Mirsky measure
with another one, and that the measure of maximal entropy,
$\mu^{}_{\max}$, is even a direct product with a binary Bernoulli
measure (for $p=\frac{1}{2}$); see \cite{KPL} for the general
situation. Also, the eigenfunctions essentially do not change, but
only span a subspace of the full Hilbert space
$L^{2} (\XX_2, \mu^{}_{\max})$, thus admitting a countable number of
Lebesgue spectral measures in addition. In view of this situation,
which holds for other hereditary cases as well, it suffices to discuss
the systems for the Mirsky measure, which has the advantage that all
means are taken with respect to the same natural averaging sequence.

\smallskip

Two obvious lines of generalisation depart from here. In one
dimension, one can view $\XX^{}_{2}$ as a $\cB$-free system where
$\cB$ consists of the squares of all rational primes. Here, the tied
density of $V^{}_{2}$ is a very natural quantity in the following
sense. If we only demand square-freeness with respect to the first $m$
primes, we get a periodic point set $V^{(m)}_{2}$ whose density then
exists uniformly and is a strict upper bound for the (upper) density
of $V^{}_{2}$. In the limit as $m\to\infty$, one recovers the value
$6/\pi^2$, which is thus the largest possible value of the density of
$V^{}_{2}$ with respect to any averaging sequence.  It is
attained along $(I_n)^{}_{n\in\NN}$. This approach also
shows why $V^{}_{2}$ is limit-periodic in the mean, in the sense
of mean Besicovitch almost periodicity \cite{LSS}, via a filtration
limit of lattice-periodic point sets, giving an extra
interpretation of its pure-point diffraction.  Now, one could replace
the forbidden set $\cB$ by other infinite sets, which can change a
lot. This has a long history, and has recently attracted interesting
investigations; see \cite{Abda,DKKL,KKL,KRS} and references therein.

On the other hand, the set $V$ demonstrates that there are natural
higher-dimensional generalisations in the lattice sense, some of which
have been studied in \cite{CV,BH,BBHLN,BBN}. They can also be
interpreted as $\cB$-free systems in higher dimensions, though the
primary focus until now was on sets of powers of `primes' (in the
sense of prime lattices emerging from prime ideals in number
fields). In the longer run, one should continue the work of
\cite{BBHLN} and  develop the full analogue
of $\cB$-free systems in higher dimensions, but various new and
unexpected phenomena already show up in the simpler case of hereditary
$\cB$-free systems, which includes the $k$-free integers in algebraic
number fields. Fortunately, the connection with weak model sets of 
maximal density gives some extra insight into the geometry and the 
spectrum of these systems \cite{TAO,BHS,RS,Keller}.

One of the obvious questions is whether two given systems of this
kind, say $k$-free integers in a general algebraic number field and
$\ell$-free integers in the same or in another one, are topologically
conjugate, or whether one can be a factor of the other. One would
expect this to be the big exception, but tools are needed to settle
this. One well-known invariant is the \emph{topological entropy},
which can be used to establish topological inequivalence and to rule
out the factor property in one direction. This was studied in some
detail for the $k$-free integers in quadratic number fields in
\cite{BBN}, and will be continued here for cyclotomic fields. For
hereditary systems of $k$-free integers, which are then also weak
model sets of maximal density, the topological entropy is closely
connected with the value of the Dedekind zeta function of the number
field at $k$, which is a bit surprising.

Another class of invariants emerges from the symmetries and extended
symmetries of the shift spaces $\XX$, as recently introduced in
\cite{BRY} and applied to the case at hand in \cite{BBHLN,BBN} for
quadratic number fields. Once a TDS $(\XX, \ZZ^d)$ is given, one can
define \cite{KS,BRY,Baa,Bustos}
\[
  \cS (\XX) \,  \defeq \, \cent^{}_{\Aut (\XX)} (\ZZ^d)
  \quad \text{and} \quad
  \cR (\XX) \,  \defeq \, \norm^{}_{\Aut (\XX)} (\ZZ^d) \ts ,
\]
where $\Aut (\XX)$ is the group of self-homeomorphisms of $\XX$, which is
to say that we consider the automorphism group in the sense of Smale.
Clearly, these two groups (up to isomorphism) are topological
invariants as well \cite{KS,KKS}. While they do not compose a
particularly strong invariant, they are of independent interest
because they capture some intrinsic symmetries of the systems. For the
visible lattice points of $\ZZ^d$, with $d\geqslant 2$, and its hull,
$\XX$, one finds
\[
    \cS (\XX) \, = \, \ZZ^d \quad \text{and} \quad
    \cR (\XX) \, = \, \ZZ^d \rtimes \GL (d, \ZZ) \ts ,
\]
which shows the maximal extension for the case that the centraliser is
trivial; compare \cite{BBHLN}, which is an extension of \cite{Mentzen}, or
the alternative approach by \cite{Keller}. The key feature here is
that $\XX$ is hereditary; see \cite{BBHLN} for the general
argument. Similarly, for the $k$-free integers in an arbitrary
quadratic number field $K$ with ring of integers $\cO$, one always
gets $\cS = \ZZ^2$ and $\cR = \cS \rtimes \cE$ with the extension
group $\cE \simeq \cO^{\times} \rtimes \Gal (K/\QQ)$, as was recently
proved in \cite{BBN}. Here, we extend this result to the case of
arbitrary cyclotomic fields, which are still Galois extensions of
$\QQ$.

It turns out that, in all cases so far, the final result is
structurally the same, which suggests that it might hold more
generally; see \cite{Fabian} for recent progress. It is interesting 
that we get trivial centralisers in so
many examples with positive topological entropy. This shows that this
phenomenon is not restricted to low complexity shifts \cite{Don}, as
was already noticed in \cite{BRY}.

\smallskip

The paper is organised as follows. In Section~\ref{sec:prelim}, we
begin with some background material, where we also recall our notions
and the construction of the number-theoretic dynamical systems. This
approach follows the classic philosophy of \cite{Fuerst} to study
number-theoretic questions in a dynamical setting. Here, we construct
the shift spaces of interest as the orbit closure of the $k$-free
points under the lattice translation action in a Minkowski
embedding. For easier compatibility with symbolic dynamics, we choose
a Cartesian embedding version.

In Section~\ref{sec:entropy}, we determine the topological entropy and
derive some immediate consequences. This is followed by a brief
recollection and discussion of the spectral properties of the $k$-free
dynamical systems. They have pure-point dynamical spectrum, but
trivial topological point spectrum (Theorem~\ref{thm:spectrum}).

In a brief intermezzo, in Section~\ref{sec:inter}, we recall some
number-theoretic notions and results for cyclotomic fields. The
formulation is tailored to the determination of the stabiliser of the
$k$-free integers in Section~\ref{sec:stab}. This is the technically
most demanding part, culminating in Theorem~\ref{thm:stab}. Then, in
Section~\ref{sec:groups}, we use this to determine the topological
centraliser and normaliser of the $k$-free dynamical systems of
cyclotomic integers. This is followed by some comments on possible
generalisations and other topological invariants of interest.

\section{Preliminaries and known results}\label{sec:prelim}

Here, we are interested in the cyclotomic fields, that is, in
$\QQ(\xi^{}_n)$ with $\xi^{}_n$ a primitive $n$-th root of unity for
some integer $n>2$ with $n \not\equiv 2 \bmod 4$.  This convention is
chosen because $\QQ (\xi^{}_{m}) = \QQ (\xi^{}_{2m})$ holds for $m$
odd, so it avoids double coverage. For general background on
cyclotomic fields, we refer to \cite{W}. The degree of the field
extension $\QQ(\xi^{}_{n})/\QQ$ is $d=\varphi(n)$, where $\varphi$ is
Euler's totient function. From here on, we use $C_n$ and $D_n$ to
denote the cyclic and the dihedral groups (of order $n$ and $2n$,
respectively), and $C_{\infty}$ and $D_{\infty}$ for their infinite
counterparts.

The ring of integers in $K=\QQ (\xi^{}_{n})$ is
$\cO_n = \ZZ[\xi^{}_{n}]$.  A rational prime $p$ is ramified in
$\cO_n$ if and only if $p \ts\ts|\ts\ts n$.  We write
$U^{}_{n} = \{ \pm \ts \xi^{\ell}_{n} : \ell \in \ZZ \} \simeq
C^{}_{\lcm (2,n)}$ for the group of roots of unity in $K$.  The unit
group $\cO_n^{\times}$ of $\cO_n$ contains the subgroup
\begin{equation}\label{eq:units}
   U^{}_{n} \times
  \bigl( \ZZ [\xi^{}_{n} + {\xi}^{-1}_{n} ]^{\times}
  \nts / \{ \pm \ts 1 \} \bigr),
\end{equation}
where $\ZZ [\xi^{}_{n} + {\xi}^{-1}_{n} ]$ is the ring of integers in
the maximal real subfield $\QQ (\xi^{}_{n} + {\xi}^{-1}_{n} )$ of $K$.
Note that the group
$\ZZ [\xi^{}_{n} + {\xi}^{-1}_{n} ]^{\times}$ is the direct product of
$\{ \pm \ts 1 \}$ and a torsion-free subgroup. We may thus view the
quotient $\ZZ [\xi^{}_{n} + {\xi}^{-1}_{n} ]^{\times} \nts / \{ \pm \ts 1 \}$
as a subgroup of $\ZZ [\xi^{}_{n} + {\xi}^{-1}_{n} ]^{\times}$ and
thus of $\cO_n^{\times}$. Then, the group \eqref{eq:units} is a
subgroup of $\cO^{\times}_{n}$, which agrees with it if $n$ is a
prime power and has index $2$ in it when $n$ is composite; see
\cite[Prop.~2.16, Thm.~4.12 and Cor.~4.13]{W}. In the latter case, one
can use $z=1-\xi^{}_{n}$ as the additional generator
\cite[Prop.~2.8]{W}.

For $n=3$ and $n=4$, the ring $\cO_n$ is a lattice in
$\CC \simeq \RR^2$, while $\cO_n $ is a dense point set of $\CC$ for
all remaining cases.  However, it becomes a lattice in $\RR^d$ under
the standard Minkowski embedding, with $d=\varphi(n)$. We shall use a
variant of this where $\cO_n$ is directly represented as $\ZZ^d$ to
facilitate the connection with symbolic dynamics later on. We call
this the \emph{Cartesian} version of Minkowski's embedding, which will
be described in detail below in Eq.~\eqref{eq:Cartesian}.

The Galois group $\Gal (K/\QQ)$ of $K$ over $\QQ$ equals the
automorphism group of the field $K$, denoted by $\Aut^{}_{\QQ} (K)$,
and one has
\begin{equation}\label{eq:cyclotomic_Galois_group}
      \Gal (K/\QQ) \, \simeq \, (\ZZ / n \ZZ)^{\times},
\end{equation} 
where $r \bmod n$ corresponds to the automorphism $\sigma^{}_{r}$
which maps $\xi^{}_{n}$ to $\xi^{r}_{n}$.  If
$n = \prod_{p\mid n} p^{a_p}$ is the prime factorisation of $n$, we
have the decomposition
\begin{equation}\label{eq:cyclotomic_Galois_group-2}
    (\ZZ / n \ZZ)^{\times} \, \simeq \,
    \prod_{p\mid n} \bigl( \ZZ / p^{a_p} \ZZ \bigr)^{\times}
\end{equation}
as a consequence of the \emph{Chinese remainder theorem}
(CRT). Now, given $n$, let
$k\geqslant 2$ be a fixed integer and let $V=V^{}_{\nts k}$ denote the
set of $k$-free integers in $\cO_n$. Here, an element $x\in\cO_n$ is
called \emph{$k$-free} if the principal ideal $(x) = x\ts \cO_n$ is
not divisible by the $k$-th power of any prime ideal in $\cO_n$. Since
the class number of a cyclotomic field is generally not $1$, except in
finitely many cases, compare \cite[Thm.~11.1]{W}, we need the
ideal-theoretic formulation.

Starting from the Cartesian embedding $V^{\prime}_{k}$ of
$V^{}_{\nts k}$, the group $\cG = \ZZ^d$ with $d=\varphi (n)$ acts on
$V'_{k}$ by translation, and the orbit closure under the local
topology gives a compact space, denoted by $\XX_k$, with a continuous
action of $\cG$ on it. So, $(\XX_k, \cG)$ is a TDS. Here, $\XX_k$ can
canonically be identified with a closed subspace of
$\{ 0,1 \}^{\ZZ^d}$, and thus be considered as a subshift. In this
sense, $(\XX_k, \cG)$ belongs to a natural class of objects from
symbolic dynamics, though its general study is still in its infancy.

If $\Aut (\XX_k )$ denotes the group of self-homeomorphisms of $\XX_k$
without further restrictions, two natural groups to consider are the
\emph{centraliser} and the \emph{normaliser} of $\cG$ in
$\Aut (\XX_k )$,
\begin{equation}\label{eq:def-cent-norm}
\begin{split}
    \cS (\XX) \, & = \, \cent^{}_{\Aut (\XX)} (\cG) \, \defeq \,
    \{ H \in \Aut (\XX_k ) : G H = H G \text{ for all } G \in \cG \}
    \quad\text{and} \\
    \cR (\XX) \, & = \, \norm^{}_{\Aut (\XX)} (\cG) \, \defeq \,
    \{ H \in \Aut (\XX_k ) : H \cG H^{-1} = \cG \} \ts .
\end{split}    
\end{equation} 
Both are topological invariants in the sense that topologically
conjugate dynamical systems have the same centraliser and normaliser
(up to group isomorphism).  In this context, we are also interested in
the \emph{stabiliser} of $V_k$, denoted by $\stab (V_k)$, which is the
monoid that consists of all $\ZZ$-linear bijections
$A \! : \, \cO_n \xrightarrow{\quad} \cO_n$ with
$A(V_k ) \subseteq V_k $.

The Gauss and the Eisenstein numbers are the two cyclotomic cases that
are also quadratic fields, so $\cG = \ZZ^2$ in both cases. Here, the
central results are known \cite{BBHLN,BBN} as follows.

\begin{fact}\label{fact:Gauss-Eisen}
  Let\/ $(\XX_{\mathrm{G}},\ZZ^2)$ denote the TDS induced by the\/
  $k$-free Gaussian integers, and\/ $\zeta^{}_{\QQ (\ii)}$ the
  Dedekind zeta function of\/ $\QQ (\ii)$. The topological entropy of
  the TDS is given by\/ $\log (2)/\zeta^{}_{\QQ (\ii)} (k)$. Further,
  the centraliser and the normaliser of the shift action are
\[
  \cS (\XX_{\mathrm{G}}) \, = \, \ZZ^2  \quad \text{and} \quad
  \cR (\XX_{\mathrm{G}}) \, = \, \cS (\XX_{\mathrm{G}}) \rtimes D_4 \ts ,
\]
where\/ $D_4$ is the symmetry group of the square and a maximal finite
subgroup of\/ $\GL (2,\ZZ)$.

Likewise, the entropy of the TDS\/ $(\XX_{\mathrm{E}},\ZZ^2)$ induced
by the\/ $k$-free Eisenstein integers via their Cartesian
representation is\/ $\log (2)/\zeta^{}_{\QQ (\varrho)} (k)$, where\/
$\varrho = \frac{1}{2}(-1 + \ii \sqrt{3}\,)$, together with
\[
  \cS (\XX_{\mathrm{E}}) \, = \, \ZZ^2  \quad \text{and} \quad
  \cR (\XX_{\mathrm{E}}) \, = \, \cS (\XX_{\mathrm{E}}) \rtimes \cE \ts ,
\]
this time with the maximal finite subgroup\/ $\cE\simeq D_6$ of\/
$\GL(2,\ZZ)$, where\/ $D_6$ is the symmetry group of the regular
hexagon. \qed
\end{fact}

It is one goal of this paper to generalise this result to all
cyclotomic fields, and to pave the way for a treatment of $k$-free
integers in general algebraic number fields.

\section{Spectral structure and topological
     entropy}\label{sec:entropy} 

Let $n>2$ with $n\not\equiv 2\bmod 4$ be fixed and set
$K=\QQ(\xi^{}_{n})$.  Let $k \geqslant 2$ be a fixed integer and
define, with $d=\varphi(n)$ and $\cO_n = \ZZ [ \xi^{}_n ]$ as before,
\begin{equation}\label{eq:Cartesian}
     V^{\prime}_{k} \, \defeq \, \bigl\{ (m^{}_{1}, \ldots, m^{}_{d}) :
    \textstyle{\sum_{i=1}^{d}} m^{}_{i} \ts \xi^{i-1}_{n} 
    \text{ is $k$-free in $\cO_n$} \bigr\} \, \subset \, \ZZ^d .
\end{equation}
This is a fairly standard \emph{Cartesian representation} of the
$k$-free integers in $K$ as a subset of $\ZZ^d$. It harvests the fact
that the units $\{ 1, \xi^{}_{n}, \ldots, \xi^{d-1}_{n} \}$ are
linearly independent over $\ZZ$ and can thus serve as a
  $\ZZ$-basis of $\cO_n$ by \cite[Prop.~I.10.2]{Neukirch}.  Indeed,
the embedding
\[
\begin{split}
    \cO^{}_{n} \, = \, \ZZ [\xi^{}_{n}] 
    \, & \lhook\joinrel\xrightarrow{\,\,\iota\;\,} \, \ZZ^d \\
    x = \sum_{i=1}^{d} m^{}_{i} \ts \xi^{i-1}_{n}
    \, & \longmapsto \, \iota (x) = (m^{}_{1} , \ldots , m^{}_{d} )
\end{split}    
\]
is a $\ZZ$-module isomorphism.  Via $\iota$, each ideal $\fa$ in
$\cO_n$ is bijectively mapped to $\vG_{\fa} = \iota (\fa)$, which is a
sublattice of $\ZZ^d$ of index
\[
     [\ZZ^d : \vG_{\fa}] \, = \, \No (\fa) \, = \, [\cO_n : \fa ] \ts ,
\]
where $\No$ is the absolute norm on $K$.

If $\cP$ denotes the set of prime ideals in $\cO_n$, we clearly have
\begin{equation}\label{eq:limit-periodic}
     V'_k \, = \, \ZZ^d \setminus \bigcup_{\fp\in\cP}\vG_{\fp^k} \ts .
\end{equation}
This shows that $V'_k$ can be seen as a limit of periodic point sets
by first removing from $\ZZ^d$ the lattices $\vG_{\fp^k}$ with all
$\fp$ of norm $\leqslant N$, which is still a lattice-periodic subset
of $\ZZ^d$, and then sending $N\to\infty$. Here, the limit is
initially taken in the local topology \cite{TAO}, which is a metric
topology and coincides with the product topology of the configuration
space $\{ 0,1 \}^{\ZZ^d}$ upon identifying its elements with the
corresponding subsets of $\ZZ^d$. As we shall see,
Eq.~\eqref{eq:limit-periodic} implies that the set $V'_k$
is a limit-periodic point set with a strong spectral structure.

This observation can be put into the powerful context of model sets
\cite{Meyer,Moo00,TAO} as follows. For $\fp \in \cP$, set
$H_{\fp} = \ZZ^d / \vG_{\fp^k}$, which is an Abelian group of order
$\No (\fp)^k$, and define the compact Abelian group
\[
     H \, \defeq \bigotimes_{\fp\in\cP} H_{\fp} \ts .
\]
We equip it with its natural Haar measure, normalised to $1$ and
simply written as $\vol (H) = 1$. Next, we define the
\emph{$\star$-map} from $\ZZ^d$ into $H$ by
\[
   z \in \ZZ^d \; \mapsto \;
   z^{\star} \, \defeq \, \bigl( z \bmod \vG_{\fp^k} 
   \bigr)_{\fp\in\cP} \ts .
\]
We can now rewrite the set of $k$-free points as
\begin{equation}\label{eq:V-as-star}
    V'_k \, = \, \{ z \in \ZZ^d : z^{\star} \in \vO \}
    \qquad \text{with } \, \vO \, \defeq \prod_{\fp\in\cP}
    \bigl( H_{\fp} \setminus \{ 0 \} \bigr) .
\end{equation}
Here, $\vO$ is a closed subset of $H$ of Haar measure
\[
   \vol (\vO) \, = \prod_{\fp\in\cP} 
   \frac{\No (\fp)^k - 1}{\No (\fp)^k} \, = 
   \prod_{\fp\in\cP}  \bigl( 1 - \No (\fp)^{-k} \bigr)
   \, = \, \frac{1}{  \zeta^{}_{K} (k)} \ts ,
\]
where $\zeta^{}_{K}$ denotes the Dedekind zeta function of the
cyclotomic field $K$.  For basic properties of the Dedekind zeta
function, we refer the reader to \cite[Ch.~VII, \S 5]{Neukirch}, and
for the above formula to \cite[Prop.~VII.5.2]{Neukirch} in particular.

This way, we have identified $V'_k$ as a weak model set in the cut and
project scheme
\begin{equation}\label{eq:CPS}
\renewcommand{\arraystretch}{1.2}\begin{array}{r@{}ccccc@{}l}
   \\  & \RR^d & \xleftarrow{\;\;\; \pi \;\;\; } 
   & \RR^d \nts\nts \times \nts\nts H & 
   \xrightarrow{\;\: \pi^{}_{\text{int}} \;\: } & H & \\
   & \cup & & \cup & & \cup & \hspace*{-1.5ex} 
   \raisebox{1pt}{\text{\footnotesize dense}} \\
   & \ZZ^d & \xleftarrow{\;\ts 1-1 \;\ts } &
     \cL  & \xrightarrow{ \qquad } &\pi^{}_{\text{int}} (\cL) & \\
   & \| & & & & \| & \\
   & L & \multicolumn{3}{c}{\xrightarrow{\qquad\quad\quad
    \,\,\,\star\!\! \qquad\quad\qquad}} 
   &  {L_{}}^{\star\nts}  & \\ \\
\end{array}\renewcommand{\arraystretch}{1}
\end{equation}
with the lattice $\cL = \{ (z, z^{\star}) : z \in \ZZ^d \}$; see
\cite{Moo00,TAO} for background. It is called \emph{weak} because the
set $\vO$ has no interior and thus consists of boundary only, so
$\vol (\vO) = \vol (\partial \vO)$. The analysis of weak model sets,
in contrast to regular ones, requires more care because many limits
fail to be uniform, hence require the specification of averaging
sequences \cite{BHS}.

The set $V'_k$ possesses a natural density, as defined by taking the
limit of counting its points, per unit volume in $\RR^d$, in a centred
ball of increasing radius.

\begin{fact}\label{fact:nat-dens}
  The natural density of\/ $V'_{k}$ exists, and is given by
\[
   \dens (V'_k) \, \defeq \, 
   \lim_{r\to\infty} \frac{\card \bigl( V'_k \cap B^{}_r (0) \bigr)}
   {\vol \bigl( B^{}_r (0) \bigr)} \, = \, 1/\zeta^{}_{K} (k) \, = \,
    \vol (\vO) \ts .
\]
\end{fact}

\begin{proof} 
    This is a special case of \cite[Thm.~6.17]{DL}.  In their
    notation, take $D = \cO_n$, label the prime ideals as $\fp_i$,
    with $i \in \NN$, and set $\fa_i = \fp_i^k$. Then,
    $\Omega(\fp_i^k) = k > 1$ (as required for an application of
    \cite[Prop.~6.14]{DL}), $X = V_k$, and the `asymptotic density' of
    $V_k$ coincides with the natural density of $V'_k$, as the latter
    does not depend on which kind of balls we use, as long as they
    are centred; compare the last two sections of \cite{BMP}.
\end{proof}

This relation can easily be derived from
Eq.~\ref{eq:limit-periodic}. Note, however, that the density notion
requires some care, because the set $V'_k$ contains holes of arbitrary
size. This forces to specify the averaging sequence, and the density
cannot exist uniformly with respect to the location of the balls. With
the centres fixed as here, we speak of \emph{tied} density; see
\cite{BMP} for details.  Via Fact~\ref{fact:nat-dens}, we see that
$V'_k$ is actually a weak model set of maximal density in the sense of
\cite{BHS}, which has the consequence that $V'_k$ is a point set with
pure-point diffraction spectrum; see \cite{TAO} for the details of
this notion and \cite{BMP,PH,BH} for explicit examples.  The support
of the diffraction measure is a subgroup of $\RR^d$.  This equals the
\emph{dynamical spectrum}, which is an important invariant in the
measure-theoretic setting of the Halmos--von Neumann theorem.

The set $V'_k$ is locally finite, so it contains only finitely many
local configurations or patches of a given size. Here, a \emph{patch}
is any set of the form $V'_k \cap B^{}_{r} (z)$ for some $r>0$ and
$z\in\ZZ^d$.  More generally, intersections of $V'_k$ with a compact
subset of $\RR^d$ are called \emph{clusters}.  If $\#^{}_r$ denotes
the number of patches in $V'_k$ of radius $r$, the quotient
$\log (\#^{}_r)/\vol (B_r )$ converges as $r\to\infty$. The limit is
called the \emph{patch counting entropy} of $V'_k$ and denoted by
$h_{\text{pc}}$.

\begin{lemma}\label{lem:pc-1}
  The patch counting entropy of\/ $V'_k$ exists and is given by\/
  $h_{\mathrm{pc}} = \log (2) / \zeta^{}_{K} (k)$.
\end{lemma}

\begin{proof}
  The existence is a standard subadditivity argument on the basis of
  Fekete's lemma, in complete analogy to the treatment in \cite{PH}.
  In fact, it is clear that $\log(2)\, \delta$ is a lower bound
  whenever $\delta$ is the upper density of $V'_{k}$ with respect to
  any chosen van Hove averaging sequence. Equality then emerges 
  when we use the supremum of all such upper densities, and this is
  indeed achieved by the tied density of $V'_{k}$ with respect to
  centred balls of increasing radius as averaging sequence. 

  Next, the set $V'_k$ is \emph{hereditary} in the sense that, for
  every patch in it, any subcluster of the latter also occurs as a
  patch in $V'_k$, see Proposition~\ref{prop:admissible} below.  The
  proof of this claim is a straight-forward modification of the
  arguments used in \cite{BBHLN}.  Consequently,
  $\log (2) \dens (V'_k)$ is a \emph{lower} bound to $h_{\text{pc}}$,
  which is the claimed expression from Fact~\ref{fact:nat-dens}.

  But this value is also an \emph{upper} bound for $h_{\mathrm{pc}}$
  by the argument from \cite[Rem.~4.3 and Thm.~4.5]{HR} on the entropy
  of weak model sets, which is essentially equivalent to the
  observation that the tied density of $V'_k$ is also its upper
  density as obtained via a supremum over all possible averaging
  sequences. Consequently, we get equality.
\end{proof}

Next, we turn the point set $V'_k$ into a dynamical system by defining
\begin{equation}\label{eq:def-dyn-syst}
     \XX_k \, \defeq \, \overline{\ts\ZZ^d + V'_k} \ts ,
\end{equation}
where $\ZZ^d + V'_k$ is a shorthand for the orbit
$\{ t + V'_k : t \in \ZZ^d \}$ and the closure is taken in the local
topology mentioned earlier. Here, $\XX_k$ can either be viewed as a
collection of subsets of $\ZZ^d$ or, equivalently, as a collection of
elements of the configuration space $\{ 0,1 \}^{\ZZ^d}$, also known as
the full binary shift space; see \cite{BBHLN} for the details on this
identification.  Either way, $\XX_k$ is a compact space with a
continuous action of $\ZZ^d$ on it via translations, so
$(\ts \XX_k, \ZZ^d)$ is a TDS.

Due to the complicated structure of the set $V'_k$, it is not easy to
understand the topological space $\XX_k$.  To continue, we define
$\cB = \{ \vG_{\fp^k} : \fp \in \cP \}$ and call a set
$U\subset \ZZ^d$ \emph{admissible} for $\cB$ if, for every
$\vG_{\fp^k} \in \cB$, the set $U$ meets at most $\No (\fp)^k - 1$
cosets of $\vG_{\fp}$ in $\ZZ^d$, that is, misses at least one. The
collection of all admissible subsets of $\ZZ^d$, namely
\begin{equation}\label{eq:def-A}
    \AAA_k \, = \, \{ \vL \subset \ZZ^d : 
    \card ( \vL \bmod \vG_{\fp^k} ) \leqslant \No (\fp)^k - 1
    \text{ for all } \fp \in \cP \} \ts ,
\end{equation}
is a closed and $\ZZ^d$-invariant set and thus constitutes a
\emph{subshift}.  Clearly, $V'_k \in \AAA^{}_k$, as well as
$t + V'_k \in \AAA^{}_k$ for all $t\in\ZZ^d$, and since no residue
class can be created in a limit, we clearly have
$\XX_k \subseteq \AAA_k$. Moreover, one has the following result.

\begin{prop}\label{prop:admissible}
  Let\/ $\XX_k$ and\/ $\AAA_k$ be defined as in
  \eqref{eq:def-dyn-syst} and \eqref{eq:def-A}. Then, for any fixed\/
  $n>2$ with\/ $n\not\equiv 2 \bmod 4$ and any\/ $k \geqslant 2$, one
  has\/ $\, \XX_k = \AAA_k$.  In particular, $\XX_k$ is hereditary, so
  arbitrary subsets of elements of\/ $\XX_k$ are again elements of\/
  $\XX_k$.
\end{prop} 

\begin{proof}
  While the relation $\XX_k \subseteq \AAA_k$ is clear by our above
  arguments, the converse is the non-trivial part of the statement.
  If follows via \cite[Prop.~5.2]{BBHLN}, which rests on an asymptotic
  density argument that ultimately holds as a consequence of
  Fact~\ref{fact:nat-dens}: The upper bound for the density of the
  required locator sets from the filtration along certain primes is
  indeed the density of the locator set, and hence positive. This
  step requires a tail estimate along the methods of \cite{DL}.
   
  Since subsets of admissible sets clearly remain admissible, $\XX_k$
  is hereditary.
\end{proof}

An important topological invariant of such a system is its
\emph{topological entropy}, $h_{\text{top}}$. Here, with
$K=\QQ (\xi^{}_{n})$ and $k$ as above, we have the following result.

\begin{prop}\label{prop:top-ent}
  The topological entropy of the TDS\/ $(\XX_k, \ZZ^d)$ agrees with
  the patch counting entropy of\/ $V'_k$, that is, one has\/
  $ \; h_{\mathrm{top}} \ts = h_{\mathrm{pc}} = \ts \log(2)/
  \zeta^{}_{K} (k)$.
\end{prop}

\begin{proof}
  The set $V'_k$ has finite local complexity (meaning that the
  collection of intersections $(t+V'_k) \cap C$ with $t\in\ZZ^d$ and
  any given compact $C \subset\RR^d$ is a finite set) and is
  identified with a configuration in the space $\{ 0,1
  \}^{\ZZ^d}$. Consequently, by \cite[Thm.~1 and Rem.~2]{BLR} together
  with \cite[Rem.~4.3 and 4.6]{HR}, we get
  $h_{\mathrm{top}} \ts = h_{\mathrm{pc}} $, and the claim follows
  from Lemma~\ref{lem:pc-1}.
\end{proof}

Let us mention in passing that
$ h_{\mathrm{top}} \ts = h_{\mathrm{pc}}$ can also be seen from
\cite[Sec.~5.7 and Exc.~5.22]{CC}, because $\XX_k$ is a binary shift
space in which $V_k^{\prime}$, viewed as a configuration, has a dense
orbit.

Recall that a TDS $(\YY, \cG)$ is a \emph{factor} of $(\XX,\cG)$ when
a continuous surjection $\varphi \colon \XX \longrightarrow \YY$
exists such that the diagram
\[
\begin{CD}
    \XX  @>G>>  \XX \\
    @V\varphi VV       @VV\varphi V  \\
    \YY  @>G>>  \YY
\end{CD}\smallskip
\]
is commutative for every $G\in \cG$. For $\cG=\ZZ^d$, one usually
verifies this for $d$ elements that generate the group $\cG$.  The two
systems are \emph{topologically conjugate} when $\varphi$ is a
bijection.

\begin{coro}
  Let\/ $n>2$ with\/ $n\not\equiv 2 \bmod 4$ be fixed, $d=\varphi (n)$,
  and let\/ $k > \ell > 1$ be arbitrary positive integers.  Then, the
  TDS\/ $(\XX_k ,\ZZ^d)$ is not a factor of\/ $(\XX_{\ell},
  \ZZ^d)$. In particular, the two systems can never be topologically
  conjugate.
\end{coro}

\begin{proof}
  Observe that the Dedekind zeta function $\zeta^{}_{K} (s)$ is a
  decreasing function of the real variable $s>1$, with
  $\lim_{s\to\infty} \zeta^{}_{K} (s) =1$.  Consequently, by
  Proposition~\ref{prop:top-ent}, the topological entropy of
  $(\XX_k, \ZZ^d)$ is increasing in $k$, approaching $\log (2)$ as
  $k\to\infty$.
   
  Since the topological entropy of a factor can never exceed that of
  its preimage, compare \cite[Prop.~10.1.3]{VO}, the two claims are
  obvious.
 \end{proof}

 Since $\XX_k$ is compact, there exists at least one
 translation-invariant probability measure on it. In fact, there are
 many. A particularly relevant one is the \emph{Mirsky measure},
 called $\mu_{_\mathrm{M}}$, which can be specified by attaching the
 tied, natural frequencies of the possible patches (or clusters) to
 the cylinder sets defined by them. These frequencies emerge by
 counting the number of occurrences of any fixed, finite patch (or
 cluster) of $V^{\prime}_{k}$ within $B_r (0)$, dividing by the
 $\vol \bigl( B_r (0)\bigr)$, and taking the limit as $r\to\infty$.
 As such, the Mirsky measure is well defined; see \cite{PH,BH} and
 references therein. This way,
 $\bigl( \ts \XX^{}_k, \ZZ^d, \mu_{_\mathrm{M}} \bigr)$ becomes a
 \emph{measure-theoretic dynamical system}, MTDS. The Mirsky measure
 is ergodic, compare \cite{PH,BHS}, which is to say that any
 $\ZZ^d$-invariant subsets of $\XX_k$ has measure either $0$ or
 $1$. For such systems, there exists the notion of the \emph{dynamical
   spectrum}. Here, one has the following result.

\begin{theorem}\label{thm:spectrum}
  The dynamical spectrum of the MTDS\/
  $\bigl( \ts \XX^{}_k, \ZZ^d, \mu_{_\mathrm{M}} \bigr)$ is pure
  point, which is to say that the space\/
  $L^2 \bigl( \XX^{}_k, \mu_{_\mathrm{M}} \bigr)$ possesses a basis of
  simultaneous eigenfunctions under the action of\/ $\ZZ^d$.
  
  Equivalently, the natural diffraction measure of the uniform Dirac comb on\/
  $V'_k$ is a pure-point measure.
\end{theorem}

\begin{proof}
  The pure-point diffraction result is proved, in even greater
  generality, in \cite{BHS}, and also follows from \cite{Keller}. The
  first claim then follows from \cite{BL}, via the equivalence theorem
  for the two notions of pure-point spectra.
\end{proof}

It is interesting to note that, except for the trivial eigenfunction
$f\equiv 1$, no eigenfunction is continuous. In other words, the
topological point spectrum is trivial. However, all eigenfunctions are
almost continuous in the sense that they are continuous on a subset of
$\XX_k$ of full measure under $\mu_{_\mathrm{M}}$.  This follows from
the results in \cite{BHS,Keller}.

\begin{remark}\label{rem:max-entropy}
  The measure-theoretic entropy of this MTDS vanishes \cite{PH}. This
  means that the Mirsky measure is not a measure of maximal entropy
  (which would equal the topological entropy, $h_{\text{top}}$, by the
  variational principle of ergodic theory). This is connected with the
  hereditary structure of $\XX_k$; see \cite{Peck} for the related
  analysis of the square-free integers of $\ZZ$.

  However, the measure of maximal entropy, in all our cases, is a
  product of the Mirsky measure with the Bernoulli measure of a fair
  coin toss. Further, the set of eigenfunctions (which can be obtained
  via the Fourier--Bohr coefficients of a generic element in $\XX_k$)
  is the same for both, with countably many Lebesgue components being
  added to the spectrum for the maximum entropy measure. Since this
  does not add any interesting feature, we focus on the Mirsky measure
  throughout.  \exend
\end{remark}

By the Halmos--von Neumann theorem, two ergodic dynamical systems with
pure-point spectrum are measure-theoretically isomorphic if and only
if they have the same spectrum. This gives further ways to distinguish
$k$-free systems from one another, though the topological notion seems
more useful in this case.

\section{Intermezzo: Algebraic number fields, units,
  and ramified primes}\label{sec:inter}

The further development profits from a more general perspective. So,
let $K$ be an algebraic number field of degree $d$ over $\QQ$, with
ring of integers $\cO = \cO^{}_{K}$. By Hilbert's theorem, we may
assume $K = \QQ(\theta)$, where $\theta$ is an algebraic number of
degree $d$, which is also the rank of $\cO$ as a $\ZZ$-module. It can
thus be identified with a lattice $\vG \subset \RR^d$, for instance
via its Minkowski embedding; see \cite{Neukirch, TAO} for
background. In particular, one can realise $\vG = \ZZ^d$ via a
Cartesian version, as detailed above for the cyclotomic case.  

Every principal ideal $(x) = x \cO$ of a non-zero $x \in \cO$ has a
unique factorization into prime ideals,
$(x) = \prod_{\fp} \fp^{v_{\fp}(x)}$.  Here, the $v_{\fp}(x)$ are
non-negative integers, only finitely many of which are
positive. Note that $v_{\fp}(xy) = v_{\fp}(x) + v_{\fp}(y)$ and
$v_{\fp}(x+y) \geqslant \min\left\{v_{\fp}(x), v_{\fp}(y)\right\}$
with equality whenever $v_{\fp}(x) \not= v_{\fp}(y)$.  Let $V_k$ as
before denote the set of $k$-free elements of $\cO$, with
$k\geqslant 2$.  Then, a non-zero $x \in \cO$ belongs to $V_k$ if
$v_{\fp}(x)<k$ for all (non-zero) prime ideals $\fp$ in $\cO$.  Let
$\pram$ be the set of rational primes that are ramified in $K$, which
is finite.

Now, for fixed $k$, we consider the set
\[
  W^{}_{\nts k} \, \defeq \, \{ x \in V^{}_k : v^{}_{\fp} (x) = 0
  \text{ for all } \fp \ts\ts |\ts\ts p \notin \pram \} \ts .
\]
From now on, we say that two elements $x,y \in \cO$ are
\emph{coprime}, denoted by $(x,y)=1$, if the principal ideals $(x)$
and $(y)$ have disjoint decompositions into prime ideals of $\cO$,
which is to say that the ideal generated by $x$ and $y$ is $\cO$.
Based upon the coprimality reasoning used in \cite{BBHLN,BBN}, we get
the following general property.

\begin{lemma}\label{lem:the-usual}
  Let\/ $K$ be a number field and\/ $k \geqslant 2$ an integer.
  Consider\/ $A \in \stab(V_k)$ and let\/ $p$ be a prime that does not
  ramify in\/ $K$. Then, if\/ $x \in V_k$ satisfies\/ $(p,x) = 1$, one
  also has\/ $(p, A(x)) = 1$. In particular, $A(W_k) \subseteq W_k$.
\end{lemma}

\begin{proof}
  Let $A$ be as assumed, and let $p$ be any rational prime that is not
  ramified, and thus square-free when considered as an element of
  $\cO$. Then, the condition $x\in V_k$ together with $(p,x)=1$
  implies $p^{k-1}\ts x \in V_k$, hence also
  $p^{k-1} A(x) = A(p^{k-1}x) \in V_k$ due to $A(V_k) \subseteq V_k$
  together with the $\ZZ$-linearity of $A$. But this is only possible
  if $(p, A(x))=1$.

  Now, if $x$ is an arbitrary element of $W_k\!$, we have
  $v^{}_{\fp} (x) = 0$ for every prime ideal that lies over a
  non-ramified rational prime $p$, hence $(p,x)=1$.  Then, also
  $(p, A(x))=1$, which implies $v^{}_{\fp} (A(x)) =0$. This gives
  $A(W_k)\subseteq W_k$ as claimed.
\end{proof}

At this point, in order to proceed, we need an argument that
Lemma~\ref{lem:the-usual} can be strengthened to conclude that
$A(W_k)=W_k$. In \cite{BBN}, for the class of $k$-free numbers in
quadratic number fields, this was done by viewing the level sets
$W^{}_{C} \defeq \{ x \in \cO : N(x) \in C \}$, where $N$ denotes the
field norm, as geometric objects and representing them as the
intersection of $\ZZ^2$ with conics.  Though it is possible to extend
this idea, at least to some extent, to the case at hand, this route
seems unlikely to be extendable to general algebraic number fields,
wherefore we use an alternative path of a more algebraic flavour.

\begin{prop}\label{prop:AW=W}
  Let\/ $\cO$ be the ring of integers in a number field\/ $K$. 
  Then, the following properties hold.
\begin{enumerate}\itemsep=2pt
\item If\/ $A\colon \cO \xrightarrow{\quad} \cO$ is a bijective\/
  $\ZZ$-linear map with\/ $A(W_{\nts k}) \subseteq W_{\nts k}$, one
  has\/ $A(W_{\nts k}) = W_{\nts k}$.
\item If\/ $A\colon \cO \xrightarrow{\quad} \cO$ is a bijective\/
  $\ZZ$-linear map with\/ $A(V_{\nts k}) \subseteq V_{\nts k}$, one
  has\/ $A(V_{\nts k}) = V_{\nts k}$.
\item The monoid\/ $\stab(V_k)$ is a group.
\end{enumerate}
\end{prop}

\begin{proof}
  We first observe that (2) implies (3). 
  We now let $B_k$ be either of $V_k$ or $W_k$ and prove (1)
  and (2) at the same time.
  Recall that $0 \not\in B_{\nts k}$, and let $b \in B_{\nts k}$ be
  arbitrary, but fixed. For any integer $m \geqslant 2$, the map $A$
  induces a bijection
\[
	A_m \colon  \cO / m \cO
	 \xrightarrow{\quad} \cO / m \cO \ts ,
\]
where $A (m \cO) = m (A \cO) = m \cO$ shows that the equivalence
classes are preserved.  Let $\widetilde{B}_{k}$ be the image of
$B_{\nts k}$ under the canonical projection to $\cO / m \cO$. Then,
$A_m \bigl( \widetilde{B}_{k} \bigr) \subseteq \widetilde{B}_{k}$;
but $\widetilde{B}_{k}$ is a finite set, so that indeed
$A_m \bigl( \widetilde{B}_k \bigr) = \widetilde{B}_k$. In other words,
for any integer $m \geqslant 2$, there is a $b_m \in B_k$ and an
$x_m \in \cO$ such that $A (b_m) = b + m \ts x_m$.  Since $A$ is a
surjection on $\cO$, we find a $y_m \in \cO$ such that $A(y_m) =
x_m$. The $\ZZ$-linearity of $A$ now gives
\[
	A(b_m - m\ts y_m) \, = \, A(b_m) - m \ts A(y_m)
	\, = \, b + m\ts  x_m - m\ts x_m \, = \, b \ts .
\]
Let us set $u \defeq b_m - m y_m \in \cO$, which does not depend on
$m$, because $A$ is injective. So, we have $u = A^{-1} (b)$, and now
need to show that $u \in B_k$.  To this end, let $\fp$ be any non-zero
prime ideal in $\cO$ and let $p \in \ZZ$ be the rational prime below
$\fp$. 
We choose $m = p^k$. Then, since
$b_{p^k} \in B_{\nts k} \subset V_k$ by construction, we get
$v_{\fp}(b_{p^k}) < k$, which is strictly less than
$v_{\fp}(p^k y_{p^k})$. So, we have
\[
  v_{\ts\fp}(u) \, = \, \min \bigl\{ v_{\ts\fp} (b_{\nts p^k}),
  v_{\ts\fp}(p^k y_{\nts p^k}) \bigr\} \, = \, v_{\fp}(b_{p^k}) \ts .
\]
This shows that, for any $b\in B_{\nts k}$, we have
$u = A^{-1} (b) \in B_{\nts k}$ as desired.
\end{proof}

\section{The stabiliser of $k$-free integers in cyclotomic 
   fields}\label{sec:stab}

 The imaginary quadratic fields $\QQ (\ii)$ and $\QQ (\rho)$ from
 Fact~\ref{fact:Gauss-Eisen} are examples of cyclotomic fields. In
 fact, they are precisely the two fields $\QQ (\xi^{}_{n})$ with
 $\varphi (n) = 2$, hence the only quadratic among the cyclotomic fields.
 So, we know $\stab (V_k)$ for these two cases from \cite{BBN}, as
 recalled in Fact~\ref{fact:Gauss-Eisen}. In particular, the monoid
 $\stab (V_k)$ is a group in these cases.

 The next possible value is $\varphi(n)=4$, which holds for
 $n\in \{5,8,12\}$.  The corresponding cyclotomic fields feature
 prominently in the theory of aperiodic order via planar quasicrystals
 with $n$-fold symmetry; see \cite[Sec.~2.5.2]{TAO}. An explicit
 derivation (which we omit at this point) leads to the result
 that the group structure
\[
    \stab (V^{}_k ) \,  \simeq \, \cO^{\times}_{n} \nts
    \rtimes \Aut^{}_{\QQ} \bigl( \QQ (\xi^{}_{n} ) \bigr)
    \, =  \, \cO^{\times}_{n} \nts\rtimes \Gal \bigl(
    \QQ( \xi^{}_{n})/\QQ\bigr)
\]
also holds in these three cases (which will become a special case
of Theorem~\ref{thm:stab} below).  So, we have
this structural result for an example where $n$ is a prime, for
another where $n$ is a prime power, and also for an example where $n$
is composite. This suggests that the result might hold for all
cyclotomic fields, the proof of which is the goal of this section.

Now, we consider
$W_{\nts k} = \{ x \in V_{k} : \ts x \in \fp \Rightarrow \fp \mid n
\}$, where $\fp$ runs through all prime ideals in $\cO_n$.  Note
that this is in accordance with the definition given in \S
\ref{sec:inter}, as the primes dividing $n$ are exactly those which
ramify in $\QQ(\xi_n)$, by \cite[Prop.~2.3]{W}.  Here, we get the
following coprimality result, which is a variant of
Lemma~\ref{lem:the-usual} for $V_k$ in the field $K=\QQ(\xi^{}_{n})$.
The last claim is a consequence of Proposition~\ref{prop:AW=W}.

\begin{lemma}\label{lem:coprimality_lemma}
  Let\/ $n$ and\/ $k$ be as above, let\/ $p \nmid n$ be a prime, and
  let\/ $A \in \stab (V_k)$. Then, if\/ $x\in V_k $ satisfies\/
  $(p,x)=1$, one also has\/ $\bigl( p, A(x) \bigr) = 1$. In
  particular, $A(W_k)  = W_k$.     \qed
\end{lemma}

Now, let $n$ be fixed, with $n \not \equiv 2 \bmod 4$, and let
$1 < m \, | \,\ts n$, also with $m \not \equiv 2 \bmod 4$.  Next,
consider the splitting primes
\[
   S_m \,  \defeq \, \{ \ell \text{ prime} : \ell \equiv 1 \bmod m \} 
        \, = \, \{ \ell \ne 2 : \ell \text{ splits completely in }
        \QQ (\xi^{}_{m} ) \} \ts ,
\]
where the equality follows from \cite[Cor.~I.10.4]{Neukirch}.  Let us
consider this in more detail.  Clearly, $\FF^{\times}_{\ell}$ is
cyclic of order $\ell -1 $ for $\ell \in S_m$. Since $m$ divides
$ \ell - 1$, there is an $r\in\ZZ$ such that
$\ord (r \bmod \ell) = m$.  Let $\Phi_m (x) \in \ZZ[x]$ be the $m$-th
cyclotomic polynomial, which is the minimal monic polynomial of
$\xi^{}_{m}$. It has degree $\varphi (m)$ and is explicitly given by
\[
    \Phi_m (x) \; = \! \prod_{j \in (\ZZ/m\ZZ)^{\times}}
    \! \bigl( x - \xi^{j}_{m} \bigr).
\]
For the fixed triple $r,\ell,m$ from above, this implies $\Phi_m (x)
\equiv \prod_{j \in (\ZZ/m\ZZ)^{\times}} ( x - r^j ) \bmod \ell$.

By the characterisation of split primes in cyclotomic fields
  \cite[Prop.~I.8.3]{Neukirch}, we get
\begin{equation}\label{eqn:prime-decompositon-split-case}
    (\ell) \, = \, \ell \ts \cO_m \, = \, \ell \ts \ZZ [\xi^{}_{m} ]
    \; = \! \prod_{j \in (\ZZ/m\ZZ)^{\times}}
    \! \fp^{}_{m,j} \ts ,
\end{equation}
where each
$\fp^{}_{m,j} \defeq (\ell, \xi^{}_{m} - r^{j'}) = (\ell , \xi^{j}_{m}
- r)$, with $j j' \equiv 1 \bmod m$, is a prime ideal (in the
notation of \cite{Neukirch}, take $\fp = (\ell)$, $\smallO = \ZZ$,
and $\theta= \xi_m$. Then, $p(X) = \Phi_m(X)$, the prime ideal
$(\ell)$ is always coprime with the conductor
$\mathfrak{F} = \cO_m$, and $p_j(X) = X-r^{j'}$ together with
$e_j = 1$. Our $\varphi(m)$ then is the $r$ from
\cite{Neukirch}). All these ideals are distinct, because the
rational prime $\ell$ splits completely.

For each $\ell \in S_m$, which is an odd prime, the group
$(\ZZ / \ell^{\ts 2} \ZZ)^{\times}\!$ is cyclic of order
$\ell(\ell-1)$. It follows that
\begin{equation}\label{eqn:cardinality-exp-n}
  \card \big\{ x \in \ZZ / \ell^{\ts 2} \ZZ :
  x^n \equiv 1 \bmod \ell^{\ts 2} \big\}
  \, = \, \gcd \bigl( \ell(\ell-1), n \bigr)
  \; \leqslant \, n  \ts .
\end{equation}
Now, fix a prime $q \in S_n$. We assert that we can choose an
$a^{}_{q} \in \ZZ$ with the following three properties,
\begin{itemize}\itemsep=2pt
\item[(H1)] $\ord (a^{}_{q} \bmod q^2 ) = n \ts q  \ts $;
\item[(H2)] $a^{}_{q} \equiv 0 \bmod p$ for any prime
   $p \ts\ts | \ts n \ts $;
 \item[(H3)] $a_{q}^n \not\equiv 1 \bmod \ell^{\ts 2}$ for all
   $\ell \in S_m$ and all $1< m \ts\ts | \ts\ts n$ with
   $m \not \equiv 2 \bmod 4 \ts $.
\end{itemize}

\begin{remark}\label{rem:conditions-well-def}
  Note that these conditions do not contradict each other.  First,
  since $q \in S_n$, it splits completely in $K$. In particular, it
  does not ramify and hence $q \nmid n$, while (H1) clearly implies
  that (H3) holds with $\ell = q$. Second, $p \ts\ts | \ts\ts n$
  forces $p \not\in S_n$, but it may still happen that $p \in S_m$ for
  some divisor $m$ of $n$. However, if (H2) holds, $a_q$ must be
  divisible by $p$ so that $a_p \bmod p^2$ is not even invertible.
  \exend
\end{remark}

Indeed, for a fixed $q \in S_n$, there are infinitely many
choices for $a_q$.  We first consider (H1), (H2), and (H3) for
primes $\ell$ such that $\ell^2 \leqslant n$.  As these are
conditions for finitely many primes, which do not contradict each
other by Remark~\ref{rem:conditions-well-def}, they are satisfied in
a subset of $\ZZ$ of density $\delta > 0$ by the CRT.  
For $0 < R \leqslant \infty$, we write $S(R)$ for the set
of all primes $\ell \leqslant R$ that still need to be considered,
so $S(R)$ consists of all primes $\ell \nmid qn$, $\ell \leqslant R$
such that $\ell \in S_m$ for some $m \mid n$ and $\ell^2 >n$.  We
use $F_R$ to denote the set of all $a_q \in \ZZ$ that satisfy (H1),
(H2), and (H3) for all $\ell \leqslant R$. Then, for any
$R < \infty$, the density of $F_R$ is
\begin{equation}\label{eqn:lower-D}
  \delta \prod_{\ell \in S(R)} \bigl(1 - \gcd(\ell(\ell-1),n )
  / \ell^2 \bigr) \, \geqslant \, \delta \prod_{\ell \in S(R)}
  \bigl( 1 - n/ \ell^2 \bigr) \, \geqslant \, D \, > \, 0
\end{equation}
by \eqref{eqn:cardinality-exp-n}.  Here, $D$ is a positive constant,
which does not depend on $R$, and whose existence is guaranteed by the
facts that the associated series
$\sum_{\ell \text{ prime}} n/\ell^2$ converges, and each factor in
the infinite product is positive due to the condition $n < \ell^2$.
	
Let $T>0$ be an integer. Then, for every $R < \infty$, we have
\begin{eqnarray*}
  \#\bigl\{a_q \in F_{\infty} \mid 0 \leqslant a_q \leqslant T \bigr\}
  & \geqslant &
  \#\bigl\{a_q \in F_{R} \mid 0 \leqslant a_q \leqslant T\bigr\}\\[2mm]
  & &  - \sum_{\ell \in S(\infty) \setminus S(R)}
      \#\bigl\{0 \leqslant a_q \leqslant T: \ell^2 \mid a_q^n-1\bigr\}.
\end{eqnarray*}
Note that, for fixed $R$ and $T$, the latter is indeed a finite sum,
but the right hand side of the inequality might be negative.  By
\eqref{eqn:cardinality-exp-n}, in each interval of length $\ell^2$,
there are at most $n$ integers $a_q$ such that $\ell^2 \mid
a_q^n-1$. Combining this fact with \eqref{eqn:lower-D} and the
above inequality, we obtain
\begin{eqnarray*}
  \liminf_{T \to \infty} \frac{\# \bigl\{a_q \in F_{\infty} \mid 0
  \leqslant a_q \leqslant T \bigr\}}{T} & \geqslant & 
  \lim_{T \to \infty} \frac{\# \bigl\{a_q \in F_R \mid 0
  \leqslant a_q \leqslant T \bigr\}}{T} \; - \\[2mm]
  & & \lim_{T \to \infty} \myfrac{1}{T} \sum_{\ell \in S(\infty) \setminus S(R)}
  \# \bigl\{0 \leqslant a_q \leqslant T: \ell^2 \mid a_q^n-1 \bigr\} \\[2mm]
  & \geqslant & D - \sum_{\ell \in S(\infty) \setminus S(R)}
              \myfrac{n}{\ell^2} \ts .
\end{eqnarray*}
Since $\sum_{\ell \text{ prime}} n/\ell^2$ converges, the sum
$\sum_{\ell \in S(\infty) \setminus S(R)} n/\ell^2$ tends to zero as
$R \rightarrow \infty$. So, for sufficiently large $R$, we have 
\[
  D - \sum_{\ell \in S(\infty) \setminus S(R)} \frac{n}{\ell^2} 
  \, > \, 0 \ts ,
\]
and $F_{\infty}$ thus contains a subset of positive density. In
particular, it is infinite.

\begin{lemma}\label{lem:factors_of_a_q_primepower}
  Let\/ $ 3 \leqslant \nts m \!\mid\! n$ with\/
  $m\not \equiv 2 \, \bmod 4$, and\/ $j\in
  (\ZZ/m\ZZ)^{\times}\!$. Further, let\/ $a^{}_{q} \in \ZZ$ satisfy
  hypotheses\/ \textnormal{(H1)}, \textnormal{(H2)} and\/
  \textnormal{(H3)} from above.  Then, the following statements hold.
\begin{enumerate}\itemsep=2pt
 \item If\/ $\fp$ is a prime ideal that contains\/
   $\xi^{j}_{m\vphantom{t}} - a_{q}^{n/m}$, then\/
   $\fp \nts\nmid\nts n$ and\/
   $\fp^{}_{m,j} = (\ell, \xi^{j}_{m\vphantom{t}} \nts - a_{q}^{n/m} )
   \subseteq \fp$ for some\/ $\ell \in S_m$.  In fact,
   $\fp^{}_{m,j} = \fp \cap \ZZ[\xi^{}_{m}]$ is a prime ideal in\/
   $\QQ (\xi^{}_{m})$.
 \item One has\/
   $\xi^{j}_{m\vphantom{t}} \nts - a_{q}^{n/m} \in V^{}_{2}
   \setminus\nts W^{}_{\nts k}  = V^{}_{2}
   \setminus\nts W^{}_{\nts 2}\ts$.
\end{enumerate}  
\end{lemma}

\begin{proof}
  Let $ z \ts | \ts n$ be a rational prime and
  $z\ZZ \subset \fz\subset \cO_n$ a prime ideal above it.  By property
  (H2), we have $a_q\in z\ZZ \subset \fz$ and therefore
  $\xi^j_{m\vphantom{t}} - a_q^{n / m} \equiv \xi_{m\vphantom{t}}^j
  \not\equiv 0 \mod{\fz}$. Since the unit
  $\xi_{m\vphantom{t}}^j \in \cO^{\times}_n \!$ does not belong to any
  prime ideal, we see that no prime ideal $\fp$ with
  $\fp \ts | \ts (\xi^{j}_{m\vphantom{t}} - a_q^{n / m})$ can divide
  $n$, which corresponds to the first half of property $(1)$.  For any
  such $\fp$,
  $a_q^n = (a_q^{n / m})^m \equiv (\xi_{m\vphantom{t}}^j)^m \equiv 1
  \mod{\fp}$, hence $a_q^n - 1$ is an integer divisible by the prime
  ideal $\fp$, thus divisible by the rational prime $\ell$ below
  $\fp$, so $a_q^n - 1\in \ZZ \cap \fp = (\ell)$.
	
  From this, we also see that the multiplicative order of
  $a_q^{n / m}$ modulo $\ell$ must be a divisor of $m$. We will now
  prove that this order is exactly $m$. If this were not the case,
  there would be some prime divisor $p$ of $m$ such that
  $\ord (a_q^{n / m} \bmod \ell)$ divides $m / p$. Suppose that
  $m = p^a m'$, with $p \nmid m'$. From this, we get
\[ 
   \xi_p^j - 1 \, \equiv \, (\xi_m^j)^{m / p} - a_q^{n / p} 
   \, \equiv \, (a_q^{n / m})^{m / p} - a_q^{n / p} 
   \, \equiv \, 0 \mod{\fp}, 
\]
which implies that $\fp$ divides $p$. But then $\fp$ also divides $n$,
as it is a multiple of $p$, contradicting the above statement that a
divisor of $\xi_{m\vphantom{t}}^j - a_q^{n / m}$ cannot divide $n$.
As such a prime $p$ does not exist, we must have
$\ord \bigl( a_q^{n / m}\bmod\ell \bigr)=m$.  Consequently, $m$
divides $\ell-1$, hence $\ell \in S_m$.

As $\xi_{m\vphantom{t}}^j - a_q^{n/m} \in \fp$ by the choice of this
ideal, we see that
$\fp^{}_{m, j} = (\ell, \xi_{m\vphantom{t}}^j - a_q^{n / m})$ must be
contained in $\fp$. By our previous observations after
\eqref{eqn:prime-decompositon-split-case} with $r = a_q^{n/m}$, we see
that $\fp^{}_{m, j}$ is indeed a prime ideal in $\ZZ [\xi^{}_m]$, thus
establishing (1).

Now, suppose
$\fp^2 \ts\ts | \ts \bigl( \xi_{m\vphantom{t}}^j - a_q^{n / m} \bigr)$
for some prime ideal $\fp$, and thus
$\xi_{m\vphantom{t}}^j - a_q^{n / m} \notin V_2$.  As before, this
would imply
$a_q^n = (a_q^{n / m})^m \equiv (\xi_{m\vphantom{t}}^j)^m \equiv 1
\mod{\fp^2}$, that is, $a_q^n - 1 \in \fp^2 \cap \ZZ= (\ell^{\ts 2})$
and hence $a_q^n \equiv 1 \bmod \ell^{\ts 2}$ for some $\ell \in
S_m$. This contradicts (H3).  As the prime ideal $\fp$ was arbitrary,
we conclude that $\xi_{m\vphantom{t}}^j - a_q^{n / m}$ lies in $V_2$.

For the final claim, it now suffices to show that
$\xi_{m\vphantom{t}}^j - a_q^{n / m}$ is not a unit. This is a
consequence of property (H1). Indeed, it follows from this hypothesis
that $\ord (a_q \bmod q) = n$ and thus
$\ord (a_q^{n / m} \bmod q) = m$.  By
\eqref{eqn:prime-decompositon-split-case}, with $\ell = q$, we see
that $\xi_{m\vphantom{t}}^j - a_q^{n / m}$ is contained in some prime
ideal of $\ZZ [\xi^{}_m]$ dividing $q$. In particular, it is not a
unit.
\end{proof}

The following result establishes the key properties of $A(1)$.

\begin{lemma}\label{lem:preservation_of_roots_of_1}
  If\/ $ A(W_{\nts k}) = W_{\nts k}$, the following properties hold.
   \begin{enumerate}\itemsep=2pt
   \item $A(1)$ is a unit in\/ $\cO_n$.
   \item If\/ $\xi$ is a primitive\/ $m$-th root of unity
      for some\/ $m \ts\ts | \ts\ts n$, then so is\/
      $A (\xi) \, A(1)^{-1}$.
   \end{enumerate}
\end{lemma}

\begin{proof}
  Clearly, (1) is a direct consequence of (2). Indeed, if $A(1)$ were
  not a unit, it would be divisible by some prime ideal $\fp$.  From
  (2), both $A (\xi) A (1)^{- 1}$ and its inverse belong to $\cO_n$,
  so
\[ 
   \fp \mid \nts A (1) \:\Longrightarrow\:
   \fp \mid \nts (A(\xi)  A (1)^{- 1}) A (1) = A (\xi) \ts . 
\]
This holds for any root of unity in $K$ and, as
$\cO_n = \ZZ [\xi^{}_n]$, this would imply
$ \cO_n = A( \cO_n ) \subseteq \fp$, which is absurd. Thus, we only
need to prove (2).

This property is evident for $m = 1$, where $\xi = 1$. If
$m \equiv 2 \mod{4}$, we can write $m = 2 m'$ with $m'$ odd.  In this
case, $\xi = - \xi'$, where $\xi'$ is a primitive $m'$-th root of
unity. By linearity, we then have
\[ 
     A (\xi) A (1)^{- 1} \, = \, - A (\xi') A (1)^{- 1} 
     \, = \, (-1) \bigl( A (\xi') A (1)^{- 1} \bigr), 
\]
where the first term on the right has order $2$ and the second has
order $m'$ by induction. Thus, the product must have order $2m' = m$,
as $2$ and $m'$ are coprime.  Consequently, if the result holds for
$m'$, it must hold for $m$ as well. Without loss of generality, we may
assume $m > 1$ and $m \not\equiv 2 \mod{4}$, in line with our general
setting.

Fix a prime $q \in S_n$ and choose $a_q \in \ZZ$ as above.  By
Lemma~\ref{lem:factors_of_a_q_primepower},
$\xi - a_q^{n / m} \in V_2\subseteq V_k$, but it does not belong to
$W_k$. Since $A (W_k) = W_k$, an application of
Lemmas~\ref{lem:coprimality_lemma}
and~\ref{lem:factors_of_a_q_primepower} shows that there must be some
splitting primes $\ell, \tilde\ell \in S_m$ such that
$\xi - a_q^{n / m} \in \bigl( \tilde\ell, \xi - a_q^{n / m} \bigr)
\subseteq \tilde{\fp}$, together with
$A \bigl( \xi - a_q^{n / m} \bigr) \in \fp \supseteq (\ell, \xi^j -
a_q^{n / m})$, where $\fp$ and $\tilde{\fp}$ are both prime ideals
that divide $(\ell)$ and $(\tilde\ell)$, respectively.  Linearity
implies
$A (\xi - a_q^{n / m}) = A (\xi) - a_q^{n / m} A (1) \equiv 0
\mod{\fp}$.  As $\ord \bigl( a_q^{n / m} \bmod \ell \bigr)= m$, we
must have $A (\xi)^m \equiv A (1)^m \mod{\fp}$.

Next, we claim that there are infinitely many choices of $\ell$ and
$\fp$ that satisfy these properties. To see this, suppose to the
contrary that there are only finitely many prime ideals
$\fp^{}_1, \ldots, \fp^{}_r$ with
$A (\xi)^m \equiv A (1)^m \mod{\fp_j}$ for
$1 \leqslant j \leqslant r$. Write $(\ell_j) =\fp_j \cap \ZZ$, with
$\ell_j \in S_m$ the corresponding splitting prime, and choose a prime
$q \in S_n$ strictly larger than all of the $\ell_j$. Take the integer
$a_q$ associated to $q$ which satisfies (H1), (H2) and (H3), and
suppose that $a_q$ also satisfies
\[ 
   \ord \bigl( a_q \bmod \ell_j^2 \bigr) \, = \, 2  \, \neq \, m \ts . 
\]
By the same arguments as before, we see that the set of suitable
numbers $a_q$ still contains a subset of positive density, and
is thus non-empty. As in Lemma~\ref{lem:factors_of_a_q_primepower}, we
may obtain some rational prime $\ell$ from $q$ for which
$\ord (a_q^{n / m} \bmod \ell) = m$. Further, for some prime ideal
$\fp \supseteq (\ell)$, which satisfies $\fp \cap \ZZ= (\ell)$, we
must have $A (\xi - a_q^{n / m}) \in \fp$, implying that the
congruence $A (\xi)^m \equiv A (1)^m \mod{\fp}$ holds for this $\fp$
as well. But the latter is different from all the
$\fp_1, \ldots, \fp_r$, a contradiction.

We have shown that $A(\xi)^m - A(1)^m$ is contained in
infinitely many prime ideals.  Since $\cO_n$ is a Dedekind domain,
this implies that $A (\xi)^m = A (1)^m$.  Consequently, there is an
$m$-th root of unity $\xi'$ such that $A (\xi) = \xi' A (1)$ (note
that, regardless of whether $A(1)$ is a unit or not, $A(1)^{-1}$ is
well-defined as an element of $K$, and that
$\xi' = A(\xi) A(1)^{-1} \in \cO_n$; that is, our reasoning is not
circular). If $p \ts | \ts m$, we get
$(\xi')^{m / p} \equiv (a_q^{n / m})^{m / p}_{\phantom{q}} \not\equiv
1 \mod{\fp}$, as $\ord \bigl( a_q^{n / m} \bmod \ell \bigr) = m$ by
our choice of $a_q$.  Thus, $\xi'$ is not an $(m / p)$-th root of
unity, and we conclude that $\xi' = A (\xi) A (1)^{- 1}$ must be a
primitive $m$-th root of unity.
\end{proof}

Next, we need a technical divisibility property that follows from
previous work by Mann \cite{Mann}; see also the introduction of
\cite{Zan}.

\begin{fact}\label{fact:linear_comb_of_roots_of_1}
  Let\/ $\xi^{}_n$ be a primitive\/ $n$-th root of unity, and suppose
  that
\[ 
   \alpha^{}_0 + \alpha^{}_1 \xi_n^{n_1} + \cdots 
   + \alpha^{}_{k - 1} \xi_n^{n_{k - 1}} \, = \, 0 \ts , 
\]
with rational coefficients\/ $\alpha_i$ and integer exponents\/
$n_j$. If no proper subsum vanishes, that is,
$\sum_{j \in B} \alpha^{}_{j} \xi_{n}^{n_j} \neq 0$ for every\/
$B \subsetneq \{ 0, 1, \ldots, k - 1 \}$, one has
\[ 
    \frac{n}{\gcd (n, n^{}_1, \ldots, n^{}_{k - 1})} \, \mathrel{\bigg|} 
   \prod_{p \mid k}    p \ts ,
\]
 where the product runs over all prime divisors of\/ $k$.  \qed
\end{fact}

Below, we shall need the result from
Fact~\ref{fact:linear_comb_of_roots_of_1} for sums of four terms,
where we then get that $n/\!\gcd(n, n^{}_{1}, n^{}_{2}, n^{}_{3})$ is
a divisor of $6$.

\begin{lemma}\label{lem:preserving_cyclotomic_base}
  Assume\/ $A (W_k) = W_k$ with\/ $A(1) = 1$. Let\/
  $p \ts\ts | \ts\ts n$ be a prime factor of\/ $n$ and write
  $n = p^a n'$ with\/ $p \nmid n'$ and\/ $a\in\NN$.  Let
  $\xi^{}_{p^a}$ be a primitive $p^a$-th root of unity.  If\/
  $A \bigl( {\xi^{}_{p^a}} \bigr) = \xi^{}_{p^a}$, one has\/
  $A \bigl(\xi_{p^a}^i \bigr) = \xi_{p^a}^i$ for all\/
  $0 \leqslant i < p^a$.
\end{lemma}

\begin{proof}
  We prove this by induction on $i$; the result is evidently true for
  $i = 0$ and $i = 1$ by our assumptions. So, suppose that
  $A \bigl(\xi_{p^a}^{h} \bigr) = \xi_{p^a}^{h}$ for all
  $0 \leqslant h \leqslant i$ and that $i = p^b i^{}_0$ with
  $p \nmid i^{}_0$ and $b < a$. A simple computation gives
\[ 
    \xi_{p^a}^i \, = \ts \bigl( \xi_{p^a}^{p^b} \bigr)^{i^{}_0}
    \, = \, \xi_{p^{a - b}}^{i_0} , 
\]
and $\xi_{p^a}^i$ thus is a primitive $p^{a - b}$-th root of unity. If
$p^{a - b} = 2$, we must have that $\xi_{p^a}^i = - 1$, as
$2 \nmid i^{}_0$.  This implies $\xi_{p^a}^{i + 1} = - \xi_{p^a}$, and
thus
$A \bigl( \xi_{p^a}^{i + 1} \bigr) = - A \bigl( \xi_{p^a} \bigr) = -
\xi_{p^a} = \xi_{p^a}^{i + 1}$, which is the desired result. We can
thus assume $p^{a - b} > 2$ in what follows.

As before, fix a prime $q \in S_n$ and choose $a_q \in \ZZ$ as above.
By Lemma~\ref{lem:factors_of_a_q_primepower}, we obtain that
$\xi_{p^a}^i - a_q^{n / p^{a - b}} = \xi^{i_0}_{p^{a - b}} - a_q^{n /
  p^{a - b}} \in V_2$, and this number must be coprime with
$\xi_{p^a} - 1$. Note that $p \ts \ZZ [\xi_{p^a}]$ is an ideal in
$\ZZ [\xi_{p^a}]$ that factors as
$(\xi_{p^a} - 1)^{p^{a - 1} (p - 1)}$, and that
$\QQ (\xi^{}_n) \nts / \ts \QQ (\xi^{}_{p^a})$ is unramified above
$p$.  This implies that $\xi_{p^a} - 1 \in V_2$ (in fact, it also
belongs to $W_2$). Hence,
$(\xi_{p^a} - 1) \bigl( \xi_{p^a}^i - a_q^{n / p^{a - b}} \bigr) \in
V_2 \subseteq V_k$ as well, though it does not belong to $W_k$.

Choose some $r \in \ZZ/ p^a \ZZ$ such that
$A \bigl( \xi_{p^a}^{i + 1} \bigr) = \xi_{p^a}^r$; such an $r$ exists
by Lemma~\ref{lem:preservation_of_roots_of_1}. Since $A (W_k) = W_k$,
there is an $\ell \in S_{p^{a-b}}$ such that
$A \bigl((\xi_{p^a} - 1) \bigl(\xi_{p^a}^i - a_q^{n / p^{a -
    b}}\bigr)\bigr) \in \fp$, where the prime ideal $\fp$ divides
$\bigl( \ell, \xi_{p^{a - b}}^j - a_{q^{\vphantom{i}}}^{n / p^{a - b}}
\bigr)$ for some $j \in (\ZZ/ p^{a - b} \ZZ)^{\times}$. By linearity
and our inductive assumptions, we get
\begin{align*}
  A \bigl((\xi_{p^a} - 1) \bigl(\xi_{p^a}^i - a_q^{n / p^{a - b}} 
  \bigr) \bigr) \, & = \, 
    A \bigl(\xi_{p^a}^{i + 1} - \xi_{p^a}^i - a_q^{n / p^{a - b}}
    \xi_{p^a} + a_q^{n / p^{a - b}} \bigr) \\
  & =  \, \xi_{p^a}^r - \xi_{p^a}^i - a_q^{n / p^{a - b}} \xi_{p^a}
    + a_q^{n / p^{a - b}}\\
  & \equiv \,  \xi_{p^a}^r {- \xi_{p^a}^i}  - \xi_{p^a}^{p^b j + 1}
    + \xi_{p^a}^{p^b j}  \, \equiv \,  0 \mod{\fp} \ts .
\end{align*}
By varying $q$, there are infinitely many possible choices for $\fp$,
but only finitely many for $j$; hence, there is a fixed
$j \in (\ZZ/ p^{a - b} \ZZ)^{\times}$ for which the module equivalence
above holds for infinitely many prime ideals $\fp$, and is thus an
actual equality,
\[ 
     \xi_{p^a}^r - \xi_{p^a}^i - \xi_{p^a}^{p^b j + 1} +
     \xi_{p^a}^{p^b j} \, = \, 0 \ts . 
\]
Note also that $p^b j$ and $p^b j + 1$ are coprime.  By
Fact~\ref{fact:linear_comb_of_roots_of_1}, if no proper subsum of the
above equality vanishes, we must have $p^a \mid 6$, and $p$ is either
$2$ or $3$. If $p = 2$, then $4$ must divide $p^a$ (as, in this
situation, $a \geqslant 2$ since $\xi_2 = - 1$) and thus divide $6$ in
turn, a contradiction. Hence, we must have $p = p^a = 3$, but in this
case we only have to verify that $A \bigl(\xi_3^2\bigr) = \xi_3^2$ as
$i > 1$ in the inductive step. We can verify the last equality from
$1 + \xi^{}_3 + \xi_3^2 = 0$ and linearity.
	
Consequently, we may now assume $p^a \nmid 6$, so that some subsum of
the above equality vanishes. As all four terms are non-zero, there
must be two summands that cancel each other (and the other two then
cancel each other as well); otherwise, if there was a null subsum with
three summands, the remaining term would also be $0$.  If
$r \equiv i \mod{p^a}$, so that $\xi_{p^a}^r - \xi_{p^a}^i = 0$, we
must have $p^b j + 1 \equiv p^b j \mod{p^a}$ or, equivalently,
$1 \equiv 0 \mod{p^a}$, which is absurd. The only remaining
alternatives are that $\xi_{p^a}^r$ cancels out with either
$- \xi_{p^a}^{p^b j + 1}$ or $\xi_{p^a}^{p^{b } j}$.

If $p$ is an odd prime, there is no power of $\xi^{}_{p^a}$ that
equals $- \xi_{p^a}^r$, which implies
$\xi_{p^a}^r - \xi_{p^a}^{p^b j + 1} = 0$.  This gives
$r \equiv p^b j + 1$ and $ i \equiv p^b j \mod{p^a}$, and hence
\[
        r \equiv i + \nts 1 \, \bmod{p^a}  \; \Longrightarrow  \; 
        A (\xi_{p^a}^{i + 1}) = \ts \xi_{p^a}^r  = \ts \xi_{p^a}^{i + 1} 
\]
as desired. When $p = 2$, we have a third possible case, namely,
$\xi_{2^a}^r + \xi_{2^a}^{2^b j} = 0$. Since
$\xi_{2^a}^{2^{a - 1}} = \xi_2 = - 1$, this implies
$r \equiv 2^b j + 2^{a - 1} \bmod 2^a$ and
$i \equiv 2^{a - 1} + 2^b j + 1 \mod{2^a}$ and thus
$i \equiv r + 1 \mod{2^a}$, so
$A (\xi_{2^a}^{i + 1}) = \xi_{2^a}^{i - 1}$.  But
$A \bigl(\xi_{2^a}^{i - 1}\bigr) = \xi_{2^a}^{i - 1}$ by our inductive
hypothesis. This is a contradiction because $A$ is a bijection; hence,
this third scenario is impossible.
\end{proof}

The following can be seen as a partial extension of the previous result:

\begin{lemma}\label{lem:preserving_cyclotomic_base_more_factors}
  Assume that\/ $A (W_k) = W_k$ and\/ $A (1) = 1$. Let\/
  $p \ts\ts | \ts\ts n$ be prime, write\/ $n = p^a n'$ with\/
  $p \nmid n'$, and let\/ $m$ be any divisor of\/ $n'$. Assume\/
  $A (\xi^{}_m) = \xi^{}_m$ and\/ $A (\xi^{}_{p^a}) =
  \xi^{}_{p^a}$. Then, one has\/
  $A\bigl( \xi_{p^a}^i \, \xi^{}_m \bigr) = \xi_{p^a}^i \, \xi^{}_m$
  for any\/ $0 \leqslant i < p^a$.
\end{lemma}

\begin{proof}
  If $m=1$, this follows immediately from
  Lemma~\ref{lem:preserving_cyclotomic_base}.  We now assume, without
  loss of generality, that $m \not\equiv 2 \mod{4}$ and $m>1$.

  Lemma~\ref{lem:preserving_cyclotomic_base} shows us that
  $A (\xi_{p^a}^i) = \xi_{p^a}^i$ for all $0 \leqslant i < p^a$.  Fix a
  prime $q \in S_n$ and choose an $a_q \in \ZZ$ as above.  We know by
  Lemma~\ref{lem:factors_of_a_q_primepower} that
\[ 
    \xi_{p^a}^i \bigl( \xi^{}_m - a_q^{n / m} \bigr) , \, 
    \xi^{}_m \bigl(\xi_{p^a}^i - a_q^{n / p^{a - b}} \bigr) \in V_2 
    \subseteq V_k  \quad \text{when } i = p^b i^{}_0 \text{ and } 
    p \nmid i^{}_0  \ts . 
\]
Let $A (\xi_{p^a}^i \xi_m) = (\xi_{p^a} \xi_m)^r$, where $r$ is some
element of $\ZZ/ p^a m\ZZ$; note that this $r$ must exist as a
consequence of Lemma~\ref{lem:preservation_of_roots_of_1}.  Now, by
Lemmas~\ref{lem:coprimality_lemma}
and~\ref{lem:factors_of_a_q_primepower}, we also know that
$A \bigl(\xi_{p^a}^i (\xi^{}_m - a_q^{n / m})\bigr)$ belongs to some
prime ideal $\fp$ dividing
${\xi^{j}_{m\vphantom{t}}} - a_{q\vphantom{t}}^{n / m}$ for some
$j \in (\ZZ/r\ZZ)^{\times}$, hence
\[
  A (\xi_{p^a}^i \xi^{}_m - \xi_{p^a}^i a_q^{n / m}) \, = \, \xi_{p^a}^r
  \xi_m^r - \xi_{p^a}^i a_q^{n / m} \, \equiv \, \xi_{p^a}^r \xi_m^r -
  \xi_{p^a}^i \xi_m^j \, \equiv \, 0 \quad \mod{\fp} \ts .
\]
The last congruence must be an actual equality, that is,
\[ 
      \xi_{p^a}^r \xi_m^r \, = \, \xi_{p^a}^i \xi_m^j \ts . 
\]
This can be seen by varying $q$ as in
Lemma~\ref{lem:preserving_cyclotomic_base} or, more directly, via
\cite[Lemma~2.12]{W}.  As $p \nmid m$, we must have
$r \equiv i \mod{p^a}$.  Likewise, by the same reasoning,
$A \bigl(\xi^{}_m (\xi_{p^a}^i - a_q^{n / p^{a-b }})\bigr)$ must
belong to some prime ideal $\fp'$ dividing
$\xi_{p^{a - b}}^j - a_q^{n / p^{a - b}}$, for some $j$ in
$(\ZZ/ p^{a - b} \ZZ)^{\times}$, so
\[
  A \bigl( \xi^{}_m \xi_{p^a}^i - \xi^{}_m a_q^{n / p^{a-b}} \bigr) \, = \,
  \xi_{p^a}^r \xi_m^r - \xi^{}_m a_q^{n / p^{a - b}} \ts \equiv \,
  \xi_{p^a}^r \xi_m^r - \xi_{p^a}^{p^b j} \xi^{}_m \, \equiv \, 0
  \mod{\fp'}.
\]
By the same arguments, the last congruence becomes a proper equality,
implying $r \equiv 1 \mod{m}$. We conclude that
$\xi_{p^a}^r = \xi_{p^a}^i$ and $\xi_m^r = \xi^{}_m$; together, these
two equalities show that we indeed have
$A (\xi_{p^a}^i \xi^{}_m) = \xi_{p^a}^r \xi^{r}_m = \xi_{p^a}^i
\xi^{}_m$.
\end{proof}

At this point, we can state a key result of this section as follows.

\begin{prop}\label{prop:structure_of_A}
  If\/ $A$ is a\/ $\ZZ$-linear bijection of\/ $\cO_n$ with\/
  $A (W_k) = W_k$, one has\/ $A = M_{\varepsilon} \circ \sigma$ for
  some $\varepsilon \in \cO^{\times}_{n}$ and\/
  $\sigma \in \Gal (K/\QQ)$.
\end{prop}

\begin{proof}
  By Lemma~\ref{lem:preservation_of_roots_of_1}, we have
  $A (1) \in \cO^{\times}_{n}$. We may therefore assume that
  $A (1) = 1$.

  Suppose that $n=\prod_{p \mid n} p^{a_p}$, and let $\xi_{p^{a_p}}$
  be the corresponding primitive $p^{a_p}$-th roots of unity. Once
  again, Lemma~\ref{lem:preservation_of_roots_of_1} shows that
  $A \bigl( \xi_{p^{a_p}} \bigr) = \xi_{p^{a_p}}^{i_p}$ for some
  $i_p \in (\ZZ/ p^{a_p} \ZZ)^{\times}$ and every prime
  $p \ts | \ts n$. From the decomposition of the Galois group of $K$
  given by Eqs.~\eqref{eq:cyclotomic_Galois_group} and
  \eqref{eq:cyclotomic_Galois_group-2}, we see that $\Gal (K/\QQ)$ is
  isomorphic to the direct product of the
  $ \Gal \bigl( \QQ(\xi_{p^{a_p}}) /\QQ \bigr)$, so there is a unique
  Galois automorphism $\sigma$ that maps $ \xi^{}_{p^{a_p}}$ to
  $\xi_{p^{a_p}}^{i_p}$ for every prime $p \ts | \ts n$. Thus, via
  replacing $A$ by $\sigma^{-1}\circ A$, we can assume that
  $A \bigl( \xi_{p^{a_p}} \bigr) = \xi_{p^{a_p}}$ for every prime
  divisor $p$ of $n$.

  If we can prove that these additional assumptions on $A$ imply that
  it is the identity, we are done. As the roots of unity in $K$
  generate $ \cO_n$ as a $\ZZ$-module, it is enough to show that
  $A (\xi) = \xi$ for every root of unity in $K$. Take any of them and
  suppose that $\ord (\xi) = d$; we will proceed by induction on the
  number of primes dividing $d$.

  If $d = p^{a}$ (so $d$ has only one prime divisor), then
  $\xi = \xi_{p^{a_{p}}}^i$ for some $0 \leqslant i < p^{a_p}$. In
  this case, Lemma~\ref{lem:preserving_cyclotomic_base} already
  implies that $A (\xi) = \xi$.  Otherwise, choose some prime
  $p \mid d$ and suppose $\xi = \xi_{p^{a_p}}^i \xi'$, where
  $\ord (\xi')$ divides $\ord (\xi)$ and is not divisible by
  $p$. Then, $\xi' = \xi^{}_m$ where $m$ is some divisor of $n$ that
  is coprime with $p$, and, as $\ord (\xi')$ has one prime factor less
  than $\ord (\xi)$, the induction hypothesis shows that
  $A (\xi') = \xi'$. Likewise,
  Lemma~\ref{lem:preserving_cyclotomic_base} once again implies that
  $A (\xi_{p^a}^i) = \xi_{p^a}^i$. Hence, we may employ
  Lemma~\ref{lem:preserving_cyclotomic_base_more_factors} and conclude
  that $A (\xi) = A (\xi_{p^a}^i \xi_m) = \xi_{p^a}^i \xi_m = \xi$, as
  desired. By induction, this is true for every root of unity, and $A$
  is thus the identity.
\end{proof}

At this point, we can finally determine the stabiliser of $V_k$ as
follows.

\begin{theorem}\label{thm:stab}
  Let\/ $K=\QQ (\xi^{}_{n})$ be a cyclotomic field, with\/ ring of
  integers\/ $\cO_n = \ZZ[ \xi^{}_{n}]$, and let\/ $V_k$ be the set of
  all\/ $k$-free elements of\/ $\cO_n$.  If\/ $A$ is a\/ $\ZZ$-linear
  bijection of\/ $\cO_n$ with\/ $A (V_k) \subseteq V_k$, there exist a
  unit\/ $\varepsilon \in \cO^{\times}_{n}$ and a Galois
  automorphism\/ $\sigma$ of\/ $K$ such that\/
  $ A = M_{\varepsilon} \circ \sigma$. Thus, the stabiliser of the
  set\/ $V_k$ is the group
\[
  \stab (V^{}_k) \, = \, \cO^{\times}_{n} \nts\rtimes
  \Gal (K \nts\nts / \ts \QQ) \, = \,
  \cO^{\times}_{n} \nts \rtimes \Aut^{}_{\QQ} \bigl(\QQ (\xi^{}_{n})
  \bigr),
\]
where the groups appearing here are the ones described in
Section~$\ref{sec:prelim}$.
\end{theorem}

\begin{proof}
  Since $A(V_k) \subseteq V_k$, we have $A(W_k) = W_k$ by
  Lemma~\ref{lem:the-usual} and Proposition~\ref{prop:AW=W}.
  Consequently, we get $A = M_{\varepsilon} \circ \sigma$ with
  $\varepsilon \in \cO^{\times}_n$ and
  $\sigma \in \Gal (K \nts\nts / \ts \QQ)$ from
  Proposition~\ref{prop:structure_of_A}.
 
  Since every $\sigma \in \Gal (K \nts\nts / \ts \QQ)$ is a field
  automorphism, it is easy to check that, for any unit $\varepsilon$,
  one has
  $M_{\varepsilon} \circ \sigma = \sigma \circ
  M_{\sigma^{-1}(\varepsilon)}$ and
  $(M_{\varepsilon} \circ \sigma)^{-1} = \sigma^{-1} \circ
  M_{\varepsilon^{-1}} = M_{\sigma^{-1} (\varepsilon^{-1})} \circ
  \sigma^{-1}$. With this, one sees that $\stab (V_k)$ is indeed a
  group, and one has
 \[
      \sigma \circ M_{\varepsilon} \circ \sigma^{-1} \, = \, 
      M_{\sigma(\varepsilon)} \ts .
 \]
 This also implies
 $(M_{\varepsilon} \circ \sigma )\circ ( M_{\varepsilon'} \circ
 \sigma' ) = M_{\varepsilon\ts \sigma(\varepsilon')} \circ (\sigma
 \circ \sigma')$, which establishes the semi-direct product structure
 as claimed.
\end{proof}

\section{Extended symmetries and further
     invariants}\label{sec:groups}

The action of $\cG\defeq \ZZ^d$ on $\XX = \XX^{}_{V'}$ is faithful. 
We now  consider the groups
\[
    \cS (\XX) \, \defeq \, \cent^{}_{\Aut (\XX)} (\cG)
    \quad\text{and}\quad
    \cR (\XX) \, \defeq \, \norm^{}_{\Aut (\XX)} (\cG) \ts ,
\]
the (topological) centraliser and
normaliser of the shift space, as introduced earlier.  For $d>1$,
the normaliser contains interesting information,
in particular on non-trivial symmetries; see \cite{BRY} for details
and examples.

\begin{prop}\label{prop:trivial}
   Let\/ $k\geqslant 2$ be fixed, and let\/ $(\XX^{}_{\cB}, \ZZ^d)$
   be the\/ $\cB$-free system from 
   Proposition~\textnormal{\ref{prop:admissible}}. Then,
   the centraliser is the trivial one, $\cS (\XX^{}_{\cB})
   = \ZZ^d$, and the normaliser is of the form\/ $\cR (\XX^{}_{\cB})
   = \cS (\XX^{}_{\cB}) \rtimes \cH$, where\/ $\cH$ is isomorphic
   with a non-trivial subgroup of\/ $\GL(n,\ZZ)$.
\end{prop}

\begin{proof}
  This is a consequence of \cite[Thm.~5.3]{BBHLN}. In fact, the
  triviality (or minimality) of the centraliser employs an argument
  put forward by Mentzen \cite{Mentzen} for the subshift of
  square-free integers, which was then extended to lattice systems in
  \cite{BBHLN}.
  
  With this structure of the centraliser, a variant of the 
  Curtis--Hedlund--Lyndon theorem
  can be used to prove that any element of $\cR (\XX^{}_{\cB})$ must
  be affine, which gives the semi-direct product structure as claimed;
  see \cite{BRY,BBHLN} for details and background.
  
  That $\cH$ must be non-trivial follows from the observation that
  the unit group $\cO^{\times}\!$ always at least contains 
  the elements $\pm 1$. Via the Minkowski embedding $\iota$,
  this maps to a non-trivial subgroup of $\Aut(\ZZ^d) = \GL(d,\ZZ)$.
\end{proof}

Finally, we can wrap up the consequence of the stabiliser structure
as follows.

\begin{theorem}
  Let\/ $(\XX^{}_{\cB},\ZZ^d )$ be the\/ $\cB$-free system from
  Proposition~\textnormal{\ref{prop:admissible}}, and let\/ $K$ be a
  quadratic or a cyclotomic number field. Then, the normaliser is\/
  $\cR = \cS \rtimes \cE$ with\/ $\cS = \ZZ^d$ and\/
  $\cE \simeq \cO^{\times} \! \rtimes \Aut^{}_{\QQ} (K)$, where\/
  $\Aut^{}_{\QQ} (K) = \Gal (K/\QQ)$.  \qed
\end{theorem}

Let us close with some general comments.  It is an obvious and
interesting question to what extent this theorem generalises to other
number fields. It seems plausible that it should hold for $k$-free
number in arbitrary number fields that are Galois over $\QQ$, though
we do not see a sufficiently general mainly algebraic argument at
present. However, recent progress along a combined path via
  algebraic and analytic number theory looks promising \cite{Fabian},
  which complements our above derivation and extends it to general
  algebraic number fields.

As one possible extension of our results, let us mention that there is
no need to use the same power $k$ for all prime ideals. Indeed, if we
bring them in some order with non-decreasing norms, we can define an
integer $x\in\cO$ to be $\bs{\kappa}$-free, where
$\bs{\kappa} = (\kappa^{}_{1}, \kappa^{}_{2}, \ldots )$ with
$\kappa_i \in \NN$, if the principal ideal $(x)$ is not divisible by
$(\fp^{\kappa_i}_{i})$ for any $i\in\NN$, subject to some mild
restrictions on the $\kappa_i$ to ensure the Erd\H{o}s property. All
these cases will still be hereditary.

Above, we demonstrated how one can use known objects and properties
from algebraic number theory to construct and compute invariants of
certain number-theoretic dynamical systems. It is quite clear that
further invariants can and should be considered, starting from classic
characteristic quantities. Here, one could think of the number of
fundamental units with negative norm or, more generally, the
characteristic polynomial of the defining algebraic number in the
number field under consideration.

Likewise, objects such as class groups and other, more refined groups
should be helpful in analysing the possible factor relations among the
$k$-free dynamical systems. It would be nice to extend the partial
classification from \cite{BBN} on potential factors to a broader
setting. We would expect that factor relations (in the dynamical
sense) should be rare among these systems, due to the independence of
the underlying sets of primes.

\section*{Acknowledgements}

It is our pleasure to thank J\"{u}rgen Kl\"{u}ners for his cooperation
throughout the project. We thank Fabian Gundlach and an
anonymous referee for hints on the necessity of tail estimates, and
for various thoughtful suggestions that helped us to improve the
paper.  This work was supported by the German Research Council
(DFG), via Project A2 of the CRC TRR 358/1 (2023) -- 491392403
(Bielefeld -- Paderborn).  The second author received funding from
ANID/FONDECYT Postdoctorado 3230159 (year 2023).

\end{document}